\documentclass[a4paper,12pt,intlimits,oneside]{amsart}

\usepackage{amsmath}
\usepackage{amssymb}
\usepackage{amsfonts}
\usepackage{amsthm,amscd}

\newtheorem{thm}{Theorem}[section]
\newtheorem{prop}[thm]{Proposition}
\newtheorem{cor}[thm]{Corollary}
\newtheorem{lem}[thm]{Lemma}
\newtheorem{defn}[thm]{Definition}
\newtheorem{remark}[thm]{Remark}

\theoremstyle{definition}
\newcommand{\comment}[1]{}

\numberwithin{equation}{section}

\def\lsim{\raisebox{-1ex}{$~\stackrel{\textstyle <}{\sim}~$}}

\theoremstyle{definition}

\begin{document}
\title{Dual of Hardy-amalgam spaces and norms inequalities}
\author[Z.V.P. Abl\'e]{Zobo Vincent de Paul Abl\'e}
\address{Laboratoire de Math\'ematiques Fondamentales, UFR Math\'ematiques et Informatique, Universit\'e F\'elix Houphou\"et-Boigny Abidjan-Cocody, 22 B.P 582 Abidjan 22. C\^ote d'Ivoire}
\email{{\tt vincentdepaulzobo@yahoo.fr}}
\author[J. Feuto]{Justin Feuto}
\address{Laboratoire de Math\'ematiques Fondamentales, UFR Math\'ematiques et Informatique, Universit\'e F\'elix Houphou\"et-Boigny Abidjan-Cocody, 22 B.P 1194 Abidjan 22. C\^ote d'Ivoire}
\email{{\tt justfeuto@yahoo.fr}}


\date{}
\begin{abstract}
We characterize the dual spaces of the generalized Hardy spaces defined by replacing Lebesgue quasi-norms by Wiener amalgam ones. In these generalized Hardy spaces, we prove that some classical linear operators such as Calder\'on-Zygmund, convolution and Riesz potential operators are bounded.
\end{abstract}

\maketitle

\section{Introduction}
Let $\varphi\in\mathcal C^\infty(\mathbb R^d)$ with support on $B(0,1)$ such that $\int_{\mathbb R^d}\varphi dx=1$, where $B(0,1)$ is the unit open ball centered at $0$ and $\mathcal C^\infty(\mathbb R^d)$ denotes the space of infinitely differentiable complex valued functions on $\mathbb R^d$. For $t>0$, we denote by $\varphi_t$ the dilated function $\varphi_t(x)=t^{-d}\varphi(x/t)$, $x\in\mathbb R^d$. The Hardy-amalgam spaces $\mathcal H^{(q,p)}$ ( $0<q, p<\infty$) introduced in \cite{AbFt}, are a generalization of the classical Hardy spaces $\mathcal H^q$ in the sense that 
they are the spaces of tempered distributions $f$ such that the maximal function
\begin{equation}
\mathcal M_{\varphi}(f):=\sup_{t>0}|f\ast\varphi_t| \label{maximal}
\end{equation}
belongs to the Wiener amalgam spaces $(L^q,\ell^p)(\mathbb R^d)$. 

We recall that for $0<p,q\leq\infty$,  a locally integrable function $f$ belongs to the amalgam space $(L^q,\ell^p)$ if 
$$\left\|f\right\|_{q,p}:=\left\|\left\{\left\|f\chi_{_{Q_k}}\right\|_{q}\right\}_{k\in\mathbb{Z}^d}\right\|_{\ell^{p}}<\infty,$$ 
where $Q_k=k+\left[0,1\right)^{d}$ for $k\in\mathbb Z^d$.
Endowed with the (quasi)-norm $\left\|\cdot\right\|_{q,p}$, the amalgam space $(L^q,\ell^p)$ is a complete space, and a Banach space for $1\leq q, p\leq+\infty$. These spaces coincide with the Lebesgue spaces whenever $p=q$, the class of spaces decreases with respect to the local component and increases with respect to the global component. Hardy-amalgam spaces are equal to $(L^q,\ell^p)$ spaces for $1<p,q<\infty$ and the dual space of $(L^q,\ell^p)$ is known to be $(L^{q'},\ell^{p'})$ under the assumption that $1\leq q,p<\infty$ where $\frac{1}{q}+\frac{1}{q'}=1$ and $\frac{1}{p}+\frac{1}{p'}=1$ with the usual conventions (see \cite{BDD}, \cite{RBHS}, \cite{FSTW}, \cite{FH} and \cite{JSTW}).

We also know that the Hardy-Littlewood maximal operator $\mathfrak{M}$ defined for a locally integrable function $f$ by 
\begin{eqnarray}
\mathfrak{M}(f)(x)=\underset{r>0}\sup\ |B(x,r)|^{-1}\int_{B(x,r)}|f(y)|dy,\ \ x\in\mathbb{R}^d.\label{maxop}
\end{eqnarray} 
is bounded in $(L^q,\ell^p)$ ($1<q,p<\infty$) (see \cite{CLHS} Theorems 4.2 and 4.5). 

As for the classical Hardy spaces, not only the definition does not depend on the particular function $\varphi$, but the Hardy-amalgam spaces can be characterized in terms of grand maximal function, and we can also replace the regular function $\varphi$ by the Poisson kernel. The Hardy-Lorentz spaces $H^{q,\alpha}$ and $H^{p,s}$ defined by Wael Abu-Shammala and Alberto Torchinsky in \cite{AT}, are continuously embedded in $\mathcal H^{(q,p)}$ provided $q\leq p$, $\alpha\leq q$ and $s\leq p$. In the case $0<q\leq\min(1,p)$ and $0<p<\infty$, we proved in \cite{AbFt} that $\mathcal H^{(q,p)}$ has an atomic characterization, with atoms which are exactly those of classical Hardy space. 

It is also well known (see \cite{EMS})  that in the classical Hardy spaces Riesz transforms $R_j$ ($j=1,\ldots,d$), convolution operators such as Riesz potential and Calder\'on-Zygmund type operators are bounded in $\mathcal H^p$ under appropriate conditions.

The aim of this paper is to characterize the dual of $\mathcal H^{(q,p)}$ spaces when $0<q\leq p\leq 1$, 
and to extend some of these results and others, known in the case of classical $\mathcal{H}^q$ spaces to Hardy-amalgam spaces $\mathcal{H}^{(q,p)}$, when $0<q\leq 1$ and $q\leq p<+\infty$, just as it is done for other generalizations of $\mathcal{H}^q$ (see \cite{MBOW}, \cite{Gc}, \cite{NEYS}, \cite{NEYS1} and \cite{SAY}). For this purpose, we organize this paper as follows. 

In Section 2, we recall some properties of Wiener amalgam spaces $(L^q,\ell^p)$ and Hardy-amalgam spaces $\mathcal{H}^{(q,p)}$ obtained in \cite{AbFt} that we will need. We characterize the dual spaces of $\mathcal{H}^{(q,p)}$ ($0<q\leq p\leq 1$) in Section 3. The last section, consisting of three subsections, is devoted to some applications of the results of Sections 2 and 3. In the first two subsections, we respectively study the boundedness of Calder\'on-Zygmund and convolution operators in $\mathcal{H}^{(q,p)}$ spaces when $0<q\leq 1$ and $q\leq p<+\infty$. In the last subsection, we prove that the Riesz potential operators $I_\alpha$ extend to bounded operators from $\mathcal{H}^{(q,p)}$ to $\mathcal{H}^{(q_1,p_1)}$ when $0<\alpha<d$, $0<q\leq 1$; $q\leq p<+\infty$; $0<q_1\leq p_1<+\infty$; $1/{q_1}=1/q-\alpha/d$ and $1/{p_1}=1/p-\alpha/d$. 

Throughout the paper, we always let $\mathbb{N}=\left\{1,2,\ldots\right\}$ and $\mathbb{Z}_{+}=\left\{0,1,2,\ldots\right\}$. We use $\mathcal S := \mathcal S(\mathbb R^{d})$ to denote the Schwartz class of rapidly decreasing smooth functions equipped with the topology defined by the family of norms $\left\{\mathcal{N}_{m}\right\}_{m\in\mathbb{Z}_{+}}$, where for all $m\in\mathbb{Z}_{+}$ and $\psi\in\mathcal{S}$, $$\mathcal{N}_{m}(\psi):=\underset{x\in\mathbb R^{d}}\sup(1 + |x|)^{m}\underset{|\beta|\leq m}\sum|{\partial}^\beta \psi(x)|$$ with $|\beta|=\beta_1+\ldots+\beta_d$, ${\partial}^\beta=\left(\partial/{\partial x_1}\right)^{\beta_1}\ldots\left(\partial/{\partial x_d}\right)^{\beta_d}$ for all $\beta=(\beta_1,\ldots,\beta_d)\in\mathbb{Z}_{+}^d$ and $|x|:=(x_1^2+\ldots+x_d^2)^{1/2}$. The dual space of $\mathcal S$ is the space of tempered distributions denoted by $\mathcal S':= \mathcal S'(\mathbb R^{d})$ equipped with the weak-$\ast$ topology. If $f\in\mathcal{S'}$ and $\theta\in\mathcal{S}$, we denote the evaluation of $f$ on $\theta$ by $\left\langle f,\theta\right\rangle$. The letter $C$ will be used for non-negative constants independent of the relevant variables that may change from one occurrence to another. When a constant depends on some important parameters $\alpha,\gamma,\ldots$, we denote it by $C(\alpha,\gamma,\ldots)$. Constants with subscript, such as $C_{\alpha,\gamma,\ldots}$, do not change in different occurrences but depend on the parameters mentioned in it. We propose the following abbreviation $\mathrm{\bf A}\lsim \mathrm{\bf B}$ for the inequalities $\mathrm{\bf A}\leq C\mathrm{\bf B}$, where $C$ is a positive constant independent of the main parameters. If $\mathrm{\bf A}\lsim \mathrm{\bf B}$ and $\mathrm{\bf B}\lsim \mathrm{\bf A}$, then we write $\mathrm{\bf A}\approx \mathrm{\bf B}$. For any given (quasi-) normed spaces $\mathcal{A}$ and $\mathcal{B}$ with the corresponding (quasi-) norms $\left\|\cdot\right\|_{\mathcal{A}}$ and $\left\|\cdot\right\|_{\mathcal{B}}$, the symbol $\mathcal{A}\hookrightarrow\mathcal{B}$ means that if $f\in\mathcal{A}$, then $f\in\mathcal{B}$ and $\left\|f\right\|_{\mathcal{B}}\lsim\left\|f\right\|_{\mathcal{A}}$.

For $\lambda>0$ and a cube $Q\subset\mathbb R^{d}$ (by a cube we mean a cube whose edges are parallel to the coordinate axes), we write $\lambda Q$ for the cube with same center as $Q$ and side-length $\lambda$ times side-length of $Q$, while $\left\lfloor \lambda \right\rfloor$ stands for the greatest integer less or equal to $\lambda$. Also, for $x\in\mathbb R^{d}$ and $\ell>0$, $Q(x,\ell)$ will denote the cube centered at $x$ and side-length $\ell$. We use the same notations for balls. For a measurable set $E\subset\mathbb R^d$, we denote by $\chi_{_{E}}$ the characteristic function of $E$ and $\left|E\right|$ for its Lebesgue measure. 
 
Throughout this paper, without loss of generality and unless otherwise specified, we assume that cubes are closed and denote by $\mathcal{Q}$ the set of all cubes.

{\bf Acknowledgement.} The authors are very grateful to Aline Bonami for her support and to the referee for his meticulous reading of the manuscript.
\section{Prerequisites on Wiener amalgam and Hardy-amalgam spaces } 

\subsection{Wiener amalgam spaces $(L^q,\ell^p)$} 



Let  $0<q<1$ and $0<p\leq 1$. We have the following well known reverse Minkowski's inequality for $(L^q, \ell^p)$. We give a proof for the reader's convenience.

\begin{prop} \label{InverseHoldMink3}
Let $0<q<1$ and $0<p\leq 1$. For all finite sequence $\left\{f_n\right\}_{n=0}^m$ of elements of $(L^q, \ell^p)$, we have
\begin{eqnarray}
\sum_{n=0}^m\left\|f_n\right\|_{q,p}\leq\left\|\sum_{n=0}^m|f_n|\right\|_{q,p}. \label{InverseHoldMink4}
\end{eqnarray}
\end{prop}

\begin{proof}
Let $0<q<1$, $0<p\leq 1$ and a finite sequence $\left\{f_n\right\}_{n=0}^m$ of elements of $(L^q, \ell^p)$. 

For $p=1$, it is immediate that 
\begin{equation*}
\sum_{n=0}^m\left\|f_n\right\|_{q,1}
\leq\sum_{k\in\mathbb{Z}^d}\left\|\sum_{n=0}^m|f_n\chi_{_{Q_k}}|\right\|_q
=\left\|\sum_{n=0}^m|f_n|\right\|_{q,1},
\end{equation*}
by the reverse Minkowski's inequality in $L^q$ (see \cite{LGkos}, 1.1.5 (b), pp. 11-12). 

Let $0<p<1$ and set 
$a_n:=\left\{\left\|f_n\chi_{_{Q_k}} \right\|_q\right\}_{k\in\mathbb{Z}^d}$, for $n=0,1,\ldots,m$. We have 
\begin{eqnarray*}
\sum_{n=0}^m\left\|f_n\right\|_{q,p}&=&\sum_{n=0}^m\left\|a_n\right\|_{\ell^p}
\leq\left\|\left\{\sum_{n=0}^m\left\|f_n\chi_{_{Q_k}}\right\|_q\right\}_{k\in\mathbb{Z}^d}\right\|_{\ell^p}\\
&\leq&\left\|\left\{\left\|\sum_{n=0}^m|f_n\chi_{_{Q_k}}|\right\|_q\right\}_{k\in\mathbb{Z}^d}\right\|_{\ell^p}
=\left\|\sum_{n=0}^m|f_n|\right\|_{q,p},
\end{eqnarray*}
by the reverse Minkowski's inequality in $\ell^p$ and $L^q$. 
This completes the proof. 
\end{proof}

Notice that by monotone convergence theorem, (\ref{InverseHoldMink4}) holds for any sequence in $(L^q, \ell^p)$, for $0<q<1$ and $0<p\leq 1$.



Another very useful result is the following. 

\begin{prop}[\cite{LSUYY}, Proposition 11.12] \label{operamaxima}
Let $1<p,u\leq+\infty$ and $1<q<+\infty$. Then, for all sequence of measurable functions $\left\{f_n\right\}_{n\geq 0}$,
\begin{eqnarray*}
\left\|\left(\underset{n=0}{\overset{+\infty}\sum}\left[\mathfrak{M}(f_n)\right]^{u}\right)^{\frac{1}{u}}\right\|_{q,p}\approx\left\|\left(\underset{n=0}{\overset{+\infty}\sum}|f_n|^{u}\right)^{\frac{1}{u}}\right\|_{q,p}, 
\end{eqnarray*}
with the implicit equivalent positive constants independent of $\left\{f_n\right\}_{n\geq 0}$ .
\end{prop}

\subsection{Hardy-amalgam spaces $\mathcal{H}^{(q,p)}$} 

Let $0<q,p<+\infty$ and $\varphi\in\mathcal{S}$ with $\text{supp}(\varphi)\subset B(0,1)$ and $\int_{\mathbb{R}^d}\varphi(x)dx=1$. 
The Hardy-amalgam spaces $\mathcal{H}^{(q,p)}$ are Banach spaces whenever $1\leq q,p<+\infty$ and quasi-Banach spaces otherwise (see \cite{AbFt}, Proposition 3.8, p. 1912). 

From now on, we assume that $0<q\leq 1$ and $q\leq p<+\infty$. Also, for simplicity, we just denote $\mathcal{M}_{\varphi}$ by $\mathcal{M}_{0}$. We recall the following definition which is also the one of an atom for $\mathcal{H}^{q}$.

\begin{defn}\label{defhqpatom}
Let $1<r\leq+\infty$ and $s\geq\left\lfloor d\left(\frac{1}{q}-1\right)\right\rfloor$ be an integer. A function $\textbf{a}$ is a $(q,r,s)$-atom on $\mathbb{R}^d$ for $\mathcal{H}^{(q,p)}$ if there exists a cube $Q$ such that  
\begin{enumerate}
\item $\text{supp}(\textbf{a})\subset Q$ ,
\item $\left\|\textbf{a}\right\|_r\leq|Q|^{\frac{1}{r}-\frac{1}{q}}$ ; \label{defratom1}
\item $\int_{\mathbb{R}^d}x^{\beta}\textbf{a}(x)dx=0$, for all multi-indexes $\beta$ with $|\beta|\leq s$ . \label{defratom2}
\end{enumerate}
\end{defn}

 We denote by $\mathcal{A}(q,r,s)$ the set of all $(\textbf{a},Q)$ such that $\textbf{a}$ is a $(q,r,s)$-atom and $Q$ is the associated cube (with respect to Definition \ref{defhqpatom}). As it was proved in 
(\cite{AbFt}, Theorems 4.3 and 4.4), we have the following atomic decomposition theorem.

\begin{thm} 
\label{thafonda} 
Let $0<\eta\leq 1$, $\delta\geq\left\lfloor d\left(\frac{1}{q}-1\right)\right\rfloor$ be an integer and $f\in\mathcal S'$. Then, $f\in\mathcal H^{(q,p)}$ if and only if there exist a sequence $\left\{(\textbf{a}_n, Q^n)\right\}_{n\geq 0}$ in $\mathcal{A}(q,\infty,\delta)$ and a sequence of scalars $\left\{{\lambda}_n\right\}_{n\geq 0}$ such that 
\begin{eqnarray*}
\left\|\sum_{n\geq 0}\left(\frac{|\lambda_n|}{\left\|\chi_{_{Q^n}}\right\|_{q}}\right)^{\eta}\chi_{_{Q^{n}}}\right\|_{\frac{q}{\eta},\frac{p}{\eta}}<+\infty, 
\end{eqnarray*}
and $f=\sum_{n\geq 0}{\lambda}_n\textbf{a}_n$ in $\mathcal{S'}$ and $\mathcal{H}^{(q,p)}$.

Moreover, we have 
\begin{eqnarray*}
\left\|f\right\|_{\mathcal{H}^{(q,p)}}\approx\inf\left\|\sum_{n\geq 0}\left(\frac{|\lambda_n|}{\left\|\chi_{_{Q^n}}\right\|_{q}}\right)^{\eta}\chi_{_{Q^{n}}}\right\|_{\frac{q}{\eta},\frac{p}{\eta}}^{\frac{1}{\eta}}.
\end{eqnarray*}
where the infimum is taken over all decompositions of $f$ in sum of multiple of elements of $\mathcal A(q;\infty,\delta)$.
\end{thm}
Notice that in this theorem, one can replace $(q,\infty,\delta)$-atoms by $(q,r,\delta)$-atoms, provided that $\max(p,1)<r<\infty$ and $0<\eta<q$ (see\cite{AbFt}, Theorem 4.6).

Let $1<r\leq+\infty$ and $\delta\geq\left\lfloor d\left(\frac{1}{q}-1\right)\right\rfloor$ be an integer. 
For simplicity, we denote by $\mathcal{H}_{fin}^{(q,p)}$ the subspace of $\mathcal{H}^{(q,p)}$ consisting of finite linear combinations of $(q,r,\delta)$-atoms. Then, $\mathcal{H}_{fin}^{(q,p)}$ is a dense subspace of $\mathcal{H}^{(q,p)}$ under the quasi-norm $\left\|\cdot\right\|_{\mathcal{H}^{(q,p)}}$ (see \cite{AbFt}, Remark 4.7, pp. 1922-1923). 
 
For given $(q,r,\delta)$-atoms and $f$ belonging to the corresponding $\mathcal{H}_{fin}^{(q,p)}$ space, we put: 
\begin{eqnarray*}
\left\|f\right\|_{\mathcal{H}_{fin}^{(q,p)}}:=\inf\left\{\left\|\sum_{n=0}^m\left(\frac{|\lambda_n|}{\left\|\chi_{_{Q^n}}\right\|_{q}}\right)^{\eta}\chi_{_{Q^{n}}}\right\|_{\frac{q}{\eta},\frac{p}{\eta}}^{\frac{1}{\eta}}: f=\sum_{n=0}^m{\lambda}_n\mathfrak{a}_n,\ m\in\mathbb{Z}_{+}\right\}, 
\end{eqnarray*}
if $r=+\infty$ and $0<\eta\leq 1$ is a fix real number, and 
\begin{eqnarray*}
\left\|f\right\|_{\mathcal{H}_{fin}^{(q,p)}}:=\inf\left\{\left\|\sum_{n=0}^m\left(\frac{|\lambda_n|}{\left\|\chi_{_{Q^n}}\right\|_{q}}\right)^{\eta}\chi_{_{Q^{n}}}\right\|_{\frac{q}{\eta},\frac{p}{\eta}}^{\frac{1}{\eta}}: f=\sum_{n=0}^m{\lambda}_n\mathfrak{a}_n,\ m\in\mathbb{Z}_{+}\right\}, 
\end{eqnarray*}
if $\max\left\{p,1\right\}<r<+\infty$ and $0<\eta<q$ a fix real number. 
The infimum in both definitions, are taken over all finite decompositions of $f$ using $(q,r,\delta)$-atom $\mathfrak{a}_n$ supported on the cube $Q^n$ and $\left\|\cdot\right\|_{\mathcal{H}_{fin}^{(q,p)}}$ defines a quasi-norm on $\mathcal{H}_{fin}^{(q,p)}$.


With these definitions, we obtained the following result.

\begin{thm}[\cite{AbFt}, Theorem 4.9]  \label{thafondamfini}
Let $\delta\geq\left\lfloor d\left(\frac{1}{q}-1\right)\right\rfloor$ be an integer and $\max\left\{p,1\right\}<r\leq+\infty$. 
\begin{enumerate}
\item If $r<+\infty$ and $0<\eta<q$, then  $\left\|\cdot\right\|_{\mathcal{H}^{(q,p)}}$ and $\left\|\cdot\right\|_{\mathcal{H}_{fin}^{(q,p)}}$ are equivalent on $\mathcal{H}_{fin}^{(q,p)}$. \label{thafondamfini01}
\item If $r=+\infty$ and $0<\eta\leq 1$, then, $\left\|\cdot\right\|_{\mathcal{H}^{(q,p)}}$ and $\left\|\cdot\right\|_{\mathcal{H}_{fin}^{(q,p)}}$ are equivalent on $\mathcal{H}_{fin}^{(q,p)}\cap\mathcal{C}(\mathbb{R}^d)$, where $\mathcal{C}(\mathbb{R}^d)$ denotes the space of continuous complex functions on $\mathbb{R}^d$. \label{thafondamfini02}
\end{enumerate}
\end{thm}

Furthermore, by \cite{AbFt}, Lemma 4.10, p. 1929, if $\mathcal{H}_{fin}^{(q,p)}$ is the subspace of $\mathcal{H}^{(q,p)}$ consisting of finite linear combinations of $(q,\infty,\delta)$-atoms, then $\mathcal{H}_{fin}^{(q,p)}\cap\mathcal{C}(\mathbb{R}^d)$ is dense in $\mathcal{H}^{(q,p)}$ with respect to the quasi-norm $\left\|\cdot\right\|_{\mathcal{H}^{(q,p)}}$, since $\mathcal{H}_{fin}^{(q,p)}$ is dense in $\mathcal{H}^{(q,p)}$ in the quasi-norm $\left\|\cdot\right\|_{\mathcal{H}^{(q,p)}}$.
  
We recall the following definition of a molecule.  

\begin{defn}[\cite{AbFt}, Definition 5.1] \label{defhqpatomolec}
Let $1<r\leq+\infty$ and $\delta\geq\left\lfloor d\left(\frac{1}{q}-1\right)\right\rfloor$ be an integer. A measurable function $\textbf{m}$ is a $(q,r,\delta)$-molecule centered at a cube $Q$ if it satisfies the following conditions: 
\begin{enumerate}
\item $\left\|\textbf{m}\chi_{_{\widetilde{Q}}}\right\|_r\leq|Q|^{\frac{1}{r}-\frac{1}{q}}$ , where $\widetilde{Q}:=2\sqrt{d}Q$, \label{defratom11molec}
\item $|\textbf{m}(x)|\leq|Q|^{-\frac{1}{q}}\left(1+\frac{|x-x_Q|}{\ell(Q)}\right)^{-2d-2\delta-3}$, for all $x\notin\widetilde{Q}$, where $x_Q$ and $\ell(Q)$ respectively denote the center and the sidelength of $Q$. \label{defratom1molec}
\item $\int_{\mathbb{R}^d}x^{\beta}\textbf{m}(x)dx=0$, for all multi-indexes $\beta$ with $|\beta|\leq \delta$.\label{defratom2molec}
\end{enumerate}
\end{defn} 

We denote by $\mathcal{M}\ell(q,r,\delta)$ the set of all $(\textbf{m},Q)$ such that $\textbf{m}$ is a $(q,r,\delta)$-molecule centered at $Q$. We obtained in \cite{AbFt} the following theorem. 

\begin{thm}[\cite{AbFt}, Theorem 5.3] \label{thafondammolec2}
Let $\max\left\{p,1\right\}<r<+\infty$, $0<\eta<q$ and $\delta\geq\left\lfloor d\left(\frac{1}{q}-1\right)\right\rfloor$ be an integer. Then, for all sequences $\left\{(\textbf{m}_n, Q^n)\right\}_{n\geq 0}$ in $\mathcal{M}\ell(q,r,\delta)$ and for all sequences of scalars $\left\{{\lambda}_n\right\}_{n\geq 0}$ such that 
\begin{eqnarray*}
\left\|\sum_{n\geq 0}\left(\frac{|\lambda_n|}{\left\|\chi_{_{Q^n}}\right\|_{q}}\right)^{\eta}\chi_{_{Q^{n}}}\right\|_{\frac{q}{\eta},\frac{p}{\eta}}<+\infty,
\end{eqnarray*}
the series $f:=\sum_{n\geq 0}{\lambda}_n\textbf{m}_n$ converges in $\mathcal{S'}$ and $\mathcal{H}^{(q,p)}$, with 
\begin{eqnarray*}
\left\|f\right\|_{\mathcal{H}^{(q,p)}}\lsim\left\|\sum_{n\geq 0}\left(\frac{|\lambda_n|}{\left\|\chi_{_{Q^n}}\right\|_{q}}\right)^{\eta}\chi_{_{Q^{n}}}\right\|_{\frac{q}{\eta},\frac{p}{\eta}}^{\frac{1}{\eta}}. 
\end{eqnarray*}
\end{thm}

\section{A characterization of the dual space of $\mathcal{H}^{(q,p)}$}
\subsection{Some results about Campanato's spaces}
In this section, we consider an integer $\delta\geq\left\lfloor d\left(\frac{1}{q}-1\right)\right\rfloor$.  
Let $1<r\leq+\infty$. We denote by $L_{\mathrm{comp}}^r(\mathbb{R}^d)$, the subspace of $L^r$- functions with compact support. We set 
$$L_{\mathrm{comp}}^{r,\delta}(\mathbb{R}^d):=\left\{f\in L_{\mathrm{comp}}^r(\mathbb{R}^d):\ \int_{\mathbb{R}^d}f(x)x^{\beta}dx=0,\ |\beta|\leq\delta\right\}$$ 
and, for all cubes $Q$,
 $$L^{r,\delta}(Q):=\left\{f\in L^r(Q):\ \int_{Q}f(x)x^{\beta}dx=0,\ |\beta|\leq\delta\right\},$$
 where $L^r(Q)$ stands for the subspace of $L^r$-functions supported in $Q$. Also, $L^r(Q)$ endowed with the norm $\left\|\cdot\right\|_{L^r(Q)}$ defined by 
$$\left\|f\right\|_{L^r(Q)}:=\left(\int_{Q}|f(x)|^rdx\right)^{\frac{1}{r}}.$$
 We point out that $L_{\mathrm{comp}}^{r,\delta}(\mathbb{R}^d)$ is a dense subspace of $\mathcal{H}^{(q,p)}$. Indeed, $L_{\mathrm{comp}}^{r,\delta}(\mathbb{R}^d)=\mathcal{H}_{fin}^{(q,p)}$ (the subspace of $\mathcal{H}^{(q,p)}$ consisting of finite linear combinations of $(q,r,\delta)$-atoms) and $\mathcal{H}_{fin}^{(q,p)}$ is dense in $\mathcal{H}^{(q,p)}$. Let $f\in L_{\mathrm{loc}}^1$ and $Q$ be a cube. Then, there exists an unique  polynomial of $\mathcal{P_{\delta}}$ ($\mathcal{P_{\delta}}:=\mathcal{P_{\delta}}(\mathbb{R}^d)$ is the space of polynomial functions of degree at most $\delta$) that we denote by $P_Q^{\delta}(f)$ such that, for all $\mathfrak{q}\in\mathcal{P_{\delta}}$,
\begin{eqnarray}
\int_{Q}\left[f(x)-P_Q^{\delta}(f)(x)\right]\mathfrak{q}(x)dx=0. \label{Campanato1}
\end{eqnarray}
Relation (\ref{Campanato1}) follows from Riesz's representation theorem. From (\ref{Campanato1}), we have $P_Q^{\delta}(g)=g$, if $g\in\mathcal{P_{\delta}}$.   

\begin{remark} 
Let $1\leq r\leq+\infty$ and $f\in L_{\mathrm{loc}}^r$. We have, for all $Q\in\mathcal{Q}$,    
\begin{eqnarray}
\sup_{x\in Q}|P_Q^{\delta}(f)(x)|\leq\frac{C}{|Q|}\int_{Q}|f(x)|dx, \label{1Campanato3}
\end{eqnarray}
and
\begin{eqnarray}
\left(\int_{Q}|P_Q^{\delta}(f)(x)|^rdx\right)^{\frac{1}{r}}\leq C\left(\int_{Q}|f(x)|^rdx\right)^{\frac{1}{r}}, \label{3Campanato3}
\end{eqnarray}
where the constant $C>0$ does not depend on $Q$ and $f$. 
\end{remark}

For (\ref{1Campanato3}), see \cite{SZLU}, Lemma 4.1. Estimate (\ref{3Campanato3}) follows from (\ref{1Campanato3}). 

\begin{defn} [\cite{NEYS}, Definition 6.1 $\left(\mathcal{L}_{r,\phi,\delta}(\mathbb{R}^d)\right)$] Let $1\leq r\leq+\infty$, a function $\phi: \mathcal{Q}\rightarrow (0,+\infty)$  and $f\in L_{\mathrm{loc}}^r$. One denotes 
\begin{equation}
\left\|f\right\|_{\mathcal{L}_{r,\phi,\delta}}:=\sup_{Q\in\mathcal{Q}}\frac{1}{\phi(Q)}\left(\frac{1}{|Q|}\int_{Q}\left|f(x)-P_Q^{\delta}(f)(x)\right|^rdx\right)^{\frac{1}{r}}, \label{Campanato2}
\end{equation}
when $r<+\infty$, and
\begin{eqnarray}
\left\|f\right\|_{\mathcal{L}_{r,\phi,\delta}}:=\sup_{Q\in\mathcal{Q}}\frac{1}{\phi(Q)}\left\|f-P_Q^{\delta}(f)\right\|_{L^{\infty}(Q)}, \label{Campanato3}
\end{eqnarray}
when $r=+\infty$. Then, the Campanato space $\mathcal{L}_{r,\phi,\delta}(\mathbb{R}^d)$ is defined to be the set of all $f\in L_{\mathrm{loc}}^r$ such that $\left\|f\right\|_{\mathcal{L}_{r,\phi,\delta}}<+\infty$. One considers elements in $\mathcal{L}_{r,\phi,\delta}(\mathbb{R}^d)$ modulo polynomials of degree $\delta$ so that $\mathcal{L}_{r,\phi,\delta}(\mathbb{R}^d)$ is a Banach space. When one writes $f\in\mathcal{L}_{r,\phi,\delta}(\mathbb{R}^d)$, then $f$ stands for the representative of $\left\{f+\mathfrak{q}:\ \mathfrak{q}\ \text{ is polynomial of degree}\ \delta\right\}$.  
\end{defn}
\subsection{A dual of some Hardy-amalgam spaces}
For simplicity, we just denote $\mathcal{L}_{r,\phi,\delta}(\mathbb{R}^d)$ by $\mathcal{L}_{r,\phi,\delta}$ and define the function $\phi_{1}: \mathcal{Q}\rightarrow (0,+\infty)$ by $$\phi_{1}(Q)=\frac{\left\|\chi_{_{Q}}\right\|_{q,p}}{|Q|}\ ,$$ for all $Q\in\mathcal{Q}$. For any $T\in\left(\mathcal{H}^{(q,p)}\right)^{\ast}$, we set 
$$\left\|T\right\|:=\left\|T\right\|_{\left(\mathcal{H}^{(q,p)}\right)^{\ast}}=\sup_{\underset{\left\|f\right\|_{\mathcal{H}^{(q,p)}}\leq 1}{f\in\mathcal{H}^{(q,p)}}}|T(f)|.$$ 
Our duality result for $\mathcal{H}^{(q,p)}$, with $0<q\leq p\leq 1$, can be stated as follows. 

\begin{thm} \label{theoremdual}
Suppose that $0<q\leq p\leq 1$. Let $1<r\leq+\infty$. Then, $\left(\mathcal{H}^{(q,p)}\right)^{\ast}$ is isomorphic to $\mathcal{L}_{r',\phi_1,\delta}$ with equivalent norms, where $\frac{1}{r}+\frac{1}{r'}=1$. More precisely, we have the following assertions:  
\begin{enumerate}
\item Let $g\in\mathcal{L}_{r',\phi_1,\delta}$. Then, the mapping 
$$T_g:f\in\mathcal{H}_{fin}^{(q,p)}\longmapsto\int_{\mathbb{R}^d}g(x)f(x)dx,$$ where $\mathcal{H}_{fin}^{(q,p)}$ is the subspace of $\mathcal{H}^{(q,p)}$ consisting of finite linear combinations of $(q,r,\delta)$-atoms, extends to a unique continuous linear functional on $\mathcal{H}^{(q,p)}$ such that
\begin{eqnarray*}
\left\|T_g\right\|\leq C\left\|g\right\|_{\mathcal{L}_{r',\phi_{1},\delta}}, 
\end{eqnarray*}
where $C>0$ is a constant independent of $g$. \label{dualpoint1}
\item Conversely, for any $T\in\left(\mathcal{H}^{(q,p)}\right)^{\ast}$, there exists $g\in\mathcal{L}_{r',\phi_1,\delta}$ such that $$T(f)=\int_{\mathbb{R}^d}g(x)f(x)dx,\ \text{ for all }\ f\in\mathcal{H}_{fin}^{(q,p)},$$ and 
\begin{eqnarray*}
\left\|g\right\|_{\mathcal{L}_{r',\phi_{1},\delta}}\leq C\left\|T\right\|, 
\end{eqnarray*}
where $C>0$ is a constant independent of $T$. \label{dualpoint2}
\end{enumerate}
\end{thm}
\begin{proof}
For our proof, we borrow some ideas from \cite{MBOW}, Theorem 8.3, \cite{JGLR}, pp. 289-292; \cite{SZLU}, Theorem 4.1 and \cite{NEYS}, Theorem 7.5. We distinguish two cases: $1<r<+\infty$ and $r=+\infty$. 

Suppose that $1<r<+\infty$. We first prove Assertion (\ref{dualpoint1}). Fix $0<\eta<q$. Let $g\in\mathcal{L}_{r',\phi_1,\delta}$ , where $\frac{1}{r}+\frac{1}{r'}=1$. Consider the mapping $T_g$ defined on $\mathcal{H}_{fin}^{(q,p)}$ by $$T_g(f)=\int_{\mathbb{R}^d}g(x)f(x)dx,\ \forall\ f\in\mathcal{H}_{fin}^{(q,p)}.$$ It's easy to verify that $T_g$ is well defined and linear. Let $f\in\mathcal{H}_{fin}^{(q,p)}$. Then, there exist a finite sequence $\left\{\left(a_n,Q^n\right)\right\}_{n=0}^m$ in $\mathcal{A}(q,r,\delta)$ and a finite sequence of scalars $\left\{\lambda_n\right\}_{n=0}^m$ such that $f=\sum_{n=0}^m\lambda_n a_n$. Thus, by the vanishing condition of the atoms $a_n$, H\"older's inequality and Proposition \ref{InverseHoldMink3}, we have 
\begin{eqnarray*}
|T_g(f)|
&\leq&\sum_{n=0}^m|\lambda_n|\left\|a_n\right\|_r\left(\int_{Q^n}\left|g(x)-P_{Q^n}^{\delta}(g)(x)\right|^{r'}dx\right)^{\frac{1}{r'}}\\
&\leq&\sum_{n=0}^m\left\|\frac{|\lambda_n|}{\left\|\chi_{_{Q^n}}\right\|_q}\chi_{_{Q^n}}\right\|_{q,p}\left[\frac{1}{\phi_{1}(Q^n)}\left(\frac{1}{|Q^n|}\int_{Q^n}\left|g(x)-P_{Q^n}^{\delta}(g)(x)\right|^{r'}dx\right)^{\frac{1}{r'}}\right]\\ 
&\leq&\left(\sum_{n=0}^m\left\|\frac{|\lambda_n|}{\left\|\chi_{_{Q^n}}\right\|_q}\chi_{_{Q^n}}\right\|_{q,p}\right)\left\|g\right\|_{\mathcal{L}_{r',\phi_1,\delta}}\leq\left\|\sum_{n=0}^m\left(\frac{|\lambda_n|}{\left\|\chi_{_{Q^n}}\right\|_q}\right)^{\eta}\chi_{_{Q^n}}\right\|_{\frac{q}{\eta},\frac{p}{\eta}}^{\frac{1}{\eta}}\left\|g\right\|_{\mathcal{L}_{r',\phi_1,\delta}}.
\end{eqnarray*}
 Thus, 
\begin{eqnarray*}
|T_g(f)|\leq\left\|f\right\|_{\mathcal{H}_{fin}^{(q,p)}}\left\|g\right\|_{\mathcal{L}_{r',\phi_1,\delta}}\lsim \left\|f\right\|_{\mathcal{H}^{(q,p)}}\left\|g\right\|_{\mathcal{L}_{r',\phi_1,\delta}}, 
\end{eqnarray*}
by Theorem \ref{thafondamfini}. This shows that $g\in\left(\mathcal{H}^{(q,p)}\right)^{\ast}$ and 
\begin{eqnarray*}
\left\|g\right\|:=\left\|T_g\right\|\leq C\left\|g\right\|_{\mathcal{L}_{r',\phi_{1},\delta}}, 
\end{eqnarray*}
since $\mathcal{H}_{fin}^{(q,p)}$ is a dense subspace of $\mathcal{H}^{(q,p)}$ with respect to the quasi-norm $\left\|\cdot\right\|_{\mathcal{H}^{(q,p)}}$.
 
Now, we prove the second assertion. 
 Let $T\in\left(\mathcal{H}^{(q,p)}\right)^{\ast}$. Fix $0<\eta<q$.  Let $Q$ be a cube. We first prove that
\begin{eqnarray}
\left(\mathcal{H}^{(q,p)}\right)^{\ast}\subset\left(L^{r,\delta}(Q)\right)^{\ast}.\label{Hanbanc}
\end{eqnarray}
For any $f\in L^{r,\delta}(Q)\backslash\left\{0\right\}$, set $$a(x):=\left\|f\right\|_{L^r(Q)}^{-1}|Q|^{\frac{1}{r}-\frac{1}{q}}f(x)\ ,$$ for all $x\in\mathbb{R}^d$. Clearly, $(a,Q)\in\mathcal{A}(q,r,\delta)$. Thus, $f\in\mathcal{H}_{fin}^{(q,p)}\subset\mathcal{H}^{(q,p)}$ and 
\begin{eqnarray*}
\left\|f\right\|_{\mathcal{H}^{(q,p)}}=\left\|f\right\|_{L^r(Q)}|Q|^{\frac{1}{q}-\frac{1}{r}}\left\|a\right\|_{\mathcal{H}^{(q,p)}}\lsim |Q|^{\frac{1}{q}-\frac{1}{r}}\left\|f\right\|_{L^r(Q)}. 
\end{eqnarray*} 
Hence for all $f\in L^{r,\delta}(Q)$,  
\begin{eqnarray}
|T(f)|\leq\left\|T\right\|\left\|f\right\|_{\mathcal{H}^{(q,p)}}\lsim |Q|^{\frac{1}{q}-\frac{1}{r}}\left\|T\right\|\left\|f\right\|_{L^r(Q)}. \label{Campanato5}
\end{eqnarray}
Thus, $T$ is a continuous linear functional on $L^{r,\delta}(Q)$, with 
\begin{eqnarray*}
\left\|T\right\|_{\left(L^{r,\delta}(Q)\right)^{\ast}}:=\sup_{\underset{\left\|f\right\|_{L^r(Q)}\leq 1}{f\in L^{r,\delta}(Q)}}|T(f)|\lsim |Q|^{\frac{1}{q}-\frac{1}{r}}\left\|T\right\|.
\end{eqnarray*}
This proves (\ref{Hanbanc}). 

From (\ref{Hanbanc}), since $L^{r,\delta}(Q)$ is a subspace of $L^{r}(Q)$, by Hahn-Banach Theorem, there exists a continuous linear functional $T_Q$ which extends $T$ on $L^{r}(Q)$; namely $T_Q\in\left(L^{r}(Q)\right)^{\ast}$ and $T(f)=T_Q(f)$, for all $f\in L^{r,\delta}(Q)$. Also, by the duality $\left(L^{r}(Q)\right)^{\ast}=L^{r'}(Q)$, that is there exists $g^Q\in L^{r'}(Q)$ such that  
\begin{eqnarray*}
T_Q(f)=\int_{Q}f(x)g^Q(x)dx, 
\end{eqnarray*}
for all $f\in L^{r}(Q)$. Hence 
\begin{eqnarray}
T(f)=T_{Q}(f)=\int_{Q}f(x)g^{Q}(x)dx,
\label{Campanato6}
\end{eqnarray}
for all $f\in L^{r,\delta}(Q)$. Let $\left\{Q_n\right\}_{n\geq 1}$ be an increasing sequence of cubes which converges to $\mathbb{R}^d$; namely $Q_n\subset Q_{n+1}$, for all $n\geq 1$, and $\underset{n\geq 1}\bigcup Q_n=\mathbb{R}^d$. From the above result, it follows that for each cube $Q_n$, there exists $g^{Q_n}\in L^{r'}(Q_n)$, such that 
\begin{eqnarray}
T(f)=T_{Q_n}(f)=\int_{Q_n}f(x)g^{Q_n}(x)dx, 
\label{Campanato8}
\end{eqnarray}
for all $f\in L^{r,\delta}(Q_n)$.

Now, we construct a function $g\in L_{\mathrm{loc}}^{r'}$ such that 
\begin{eqnarray}
T(f)=\int_{Q_n}f(x)g(x)dx, \label{Campanato9}
\end{eqnarray}
for all $f\in L^{r,\delta}(Q_n)$ and all $n\geq1$. Let us first assume that $f\in L^{r,\delta}(Q_1)$. Then, 
\begin{eqnarray*}
T(f)=T_{Q_1}(f)=\int_{Q_1}f(x)g^{Q_1}(x)dx,
\end{eqnarray*}
by (\ref{Campanato8}). Moreover, $L^{r,\delta}(Q_1)\subset L^{r,\delta}(Q_2)$. Hence $f\in L^{r,\delta}(Q_2)$ and by (\ref{Campanato8}), 
\begin{eqnarray*}
T(f)=T_{Q_2}(f)=\int_{Q_2}f(x)g^{Q_2}(x)dx=\int_{Q_1}f(x)g^{Q_2}(x)dx. 
\end{eqnarray*}
Thus, for all $f\in L^{r,\delta}(Q_1)$, 
\begin{eqnarray}
\int_{Q_1}f(x)\left[g^{Q_1}(x)-g^{Q_2}(x)\right]dx=0. \label{Campanato10}
\end{eqnarray}
For any $h\in L^{r}(Q_1)$, we have $h-P_{Q_1}^{\delta}(h)\chi_{_{Q_1}}\in L^{r,\delta}(Q_1)$, by (\ref{Campanato1}). Hence 
\begin{eqnarray*}
0=\int_{Q_1}\left[h(x)-P_{Q_1}^{\delta}(h)(x)\chi_{_{Q_1}}(x)\right]\left[g^{Q_1}(x)-g^{Q_2}(x)\right]dx,
\end{eqnarray*}
for all $h\in L^{r}(Q_1)$, by (\ref{Campanato10}). But  
\begin{eqnarray*}
&&\int_{Q_1}\left[h(x)-P_{Q_1}^{\delta}(h)(x)\chi_{_{Q_1}}(x)\right]\left[g^{Q_1}(x)-g^{Q_2}(x)\right]dx\\
&=&\int_{Q_1}h(x)\left[(g^{Q_1}-g^{Q_2})(x)-P_{Q_1}^{\delta}(g^{Q_1}-g^{Q_2})(x)\right]dx
\end{eqnarray*}
(see \cite{JGLR}, pp. 290-291). Therefore, 
\begin{eqnarray}
\int_{Q_1}h(x)\left[(g^{Q_1}-g^{Q_2})(x)-P_{Q_1}^{\delta}(g^{Q_1}-g^{Q_2})(x)\right]dx=0, \label{Campanato11} 
\end{eqnarray}
for all $h\in L^{r}(Q_1)$. It follows from (\ref{Campanato11}) that 
\begin{eqnarray*}
(g^{Q_1}-g^{Q_2})(x)=P_{Q_1}^{\delta}(g^{Q_1}-g^{Q_2})(x), 
\end{eqnarray*}
for almost all $x\in Q_1$. Thus, after changing values of $g^{Q_1}$ (or $g^{Q_2}$) on a set of measure zero, we have $$(g^{Q_1}-g^{Q_2})(x)=P_{Q_1}^{\delta}(g^{Q_1}-g^{Q_2})(x),\ \text{ for all }\ x\in Q_1.$$ 
Arguing as above, we obtain 
\begin{eqnarray}
(g^{Q_n}-g^{Q_{n+1}})(x)=P_{Q_n}^{\delta}(g^{Q_n}-g^{Q_{n+1}})(x), \label{Campanato12}
\end{eqnarray}
for all $x\in Q_n$ and all $n\geq1$.

 Set  
\begin{eqnarray}
g_1(x):=g^{Q_1}(x),\ \text{ if }\ x\in Q_1 \label{Campanato13}
\end{eqnarray}
and $$g_{n+1}(x):=\left\{\begin{array}{lll}g^{Q_n}(x),&\text{ if }&x\in Q_n\ ,\\ 
g^{Q_{n+1}}(x)+P_{Q_n}^{\delta}(g^{Q_n}-g^{Q_{n+1}})(x),&\text{ if }&x\in Q_{n+1}\backslash{Q_n}\ ,\end{array}\right.$$ for all $n\geq 1$. Then, we have  
\begin{eqnarray}
g_{n+1}(x)=g^{Q_{n+1}}(x)+P_{Q_n}^{\delta}(g^{Q_n}-g^{Q_{n+1}})(x), \label{Campanato14}
\end{eqnarray}
for all $x\in Q_{n+1}$ and all $n\geq 1$, by (\ref{Campanato12}). With (\ref{Campanato13}) and (\ref{Campanato14}), we define the function $g$ on $\mathbb{R}^d$ by $$g(x)=g_n(x)\ ,\ \text{ if }\ x\in Q_n\ ,$$ for all $n\geq 1$. We have that $g\in L_{\mathrm{loc}}^{r'}$, since $g_n\in L_{\mathrm{loc}}^{r'}$, for all $n\geq 1$, by definition. Also, it's easy to see that  
\begin{eqnarray*}
\int_{Q_n}f(x)g(x)dx=T_{Q_n}(f)=T(f),
\end{eqnarray*}
for all $f\in L^{r,\delta}(Q_n)$ and all $n\geq1$, by (\ref{Campanato8}). Thus, the function $g$ satisfies (\ref{Campanato9}).

To finish, we show that $g\in\mathcal{L}_{r',\phi_1,\delta}$ and 
\begin{eqnarray}
T(f)=\int_{\mathbb{R}^d}f(x)g(x)dx, \label{Campanato15}  
\end{eqnarray}
for all $f\in\mathcal{H}_{fin}^{(q,p)}$. To prove (\ref{Campanato15}), consider $f\in\mathcal{H}_{fin}^{(q,p)}$. We have $f\in L_{\mathrm{comp}}^{r,\delta}(\mathbb{R}^d)$, since $L_{\mathrm{comp}}^{r,\delta}(\mathbb{R}^d)=\mathcal{H}_{fin}^{(q,p)}$. Thus, by the definition of the sequence $\left\{Q_n\right\}_{n\geq 1}$, there exists an integer $n\geq 1$ such that $f\in L^{r,\delta}(Q_n)$. Hence (\ref{Campanato15}) holds, by (\ref{Campanato9}). It remains to show that $g\in\mathcal{L}_{r',\phi_1,\delta}$. Let $Q$ be a cube and $f\in L^r(Q)$ such that $\left\|f\right\|_{L^r(Q)}\leq 1$. Set 
\begin{eqnarray}
a(x):=C_{f,Q,\delta,r}|Q|^{\frac{1}{r}-\frac{1}{q}}\left(f(x)-P_Q^{\delta}(f)(x)\right)\chi_{_Q}(x), \label{Campanato15bis}
\end{eqnarray}
for all $x\in\mathbb{R}^d$, with $C_{f,Q,\delta,r}:=\left(1+\left\|\left(f-P_Q^{\delta}(f)\right)\chi_{_Q}\right\|_r\right)^{-1}$. Notice that $C_{f,Q,\delta,r}^{-1}\leq 2+C$, where $C>0$ is a constant independent of $f$ and $Q$, by (\ref{3Campanato3}) and the fact that $\left\|f\right\|_{L^r(Q)}\leq1$. It is straightforward that $(a,Q)\in\mathcal{A}(q,r,\delta)$. Hence  
\begin{eqnarray*}
T(a)=\int_{\mathbb{R}^d}a(x)g(x)dx=\int_{Q}a(x)g(x)dx,
\end{eqnarray*}
by (\ref{Campanato15}). Since $T\in\left(\mathcal{H}^{(q,p)}\right)^{\ast}$, by using the vanishing condition of the atom $a$, Theorem \ref{thafondamfini}, 
it follows that 
\begin{eqnarray*}
\left|\int_{Q}a(x)\left[g(x)-P_Q^{\delta}(g)\right]dx\right|&=&\left|\int_{Q}a(x)g(x)dx\right|=|T(a)|\leq\left\|T\right\|\left\|a\right\|_{\mathcal{H}^{(q,p)}}\\
&\lsim& \left\|a\right\|_{\mathcal{H}_{fin}^{(q,p)}}\left\|T\right\|\lsim\frac{1}{\left\|\chi_{_{Q}}\right\|_q}\left\|\chi_{_{Q}}\right\|_{q,p}\left\|T\right\|.
\end{eqnarray*}
Combining this result with (\ref{Campanato15bis}) and (\ref{Campanato1}), we obtain 
\begin{eqnarray*}
\left|\int_{Q}f(x)\left[g(x)-P_Q^{\delta}(g)\right]dx\right|&\lsim&C_{f,Q,\delta,r}^{-1}|Q|^{\frac{1}{q}-\frac{1}{r}}\frac{1}{\left\|\chi_{_{Q}}\right\|_q}\left\|\chi_{_{Q}}\right\|_{q,p}\left\|T\right\|\\
&\lsim&|Q|^{-\frac{1}{r}}\left\|\chi_{_{Q}}\right\|_{q,p}\left\|T\right\|.
\end{eqnarray*} 
Thus, 
\begin{eqnarray*}
\left(\int_{Q}|g(x)-P_Q^{\delta}(g)|^{r'}dx\right)^{\frac{1}{r'}}&=&\sup_{\underset{\left\|f\right\|_{L^r(Q)}\leq 1}{f\in L^r(Q)}}\left|\int_{Q}f(x)\left[g(x)-P_Q^{\delta}(g)\right]dx\right|\\
&\lsim&|Q|^{-\frac{1}{r}}\left\|\chi_{_{Q}}\right\|_{q,p}\left\|T\right\|.
\end{eqnarray*}
From this inequality, it follows that
\begin{eqnarray*}
\frac{1}{\phi_1(Q)}\left(\frac{1}{|Q|}\int_{Q}|g(x)-P_Q^{\delta}(g)|^{r'}dx\right)^{\frac{1}{r'}}\lsim\left\|T\right\|.
\end{eqnarray*}
Therefore, $\left\|g\right\|_{\mathcal{L}_{r',\phi_{1},\delta}}\lsim\left\|T\right\|$ and $g\in\mathcal{L}_{r',\phi_{1},\delta}$. This finishes the proof 
for the case where $1<r<+\infty$.

We consider now the case where $r=+\infty$. Let $g\in\mathcal{L}_{1,\phi_{1},\delta}$. Consider the mapping $T_g$ defined on $\mathcal{H}_{fin}^{(q,p)}\cap\mathcal{C}(\mathbb{R}^d)$ by $$T_g(f)=\int_{\mathbb{R}^d}f(x)g(x)dx,\ \forall\ f\in\mathcal{H}_{fin}^{(q,p)}\cap\mathcal{C}(\mathbb{R}^d).$$ Fix $0<\eta<q$. Arguing as in the first case, we obtain
\begin{eqnarray*}
|T_g(f)|\leq\left\|f\right\|_{\mathcal{H}_{fin}^{(q,p)}}\left\|g\right\|_{\mathcal{L}_{1,\phi_1,\delta}}\lsim \left\|f\right\|_{\mathcal{H}^{(q,p)}}\left\|g\right\|_{\mathcal{L}_{1,\phi_1,\delta}},  
\end{eqnarray*}
by Theorem  \ref{thafondamfini}. This shows that $g\in\left(\mathcal{H}^{(q,p)}\right)^{\ast}$ and 
\begin{eqnarray*}
\left\|g\right\|:=\left\|T_g\right\|\leq C\left\|g\right\|_{\mathcal{L}_{1,\phi_1,\delta}}, 
\end{eqnarray*}
since $\mathcal{H}_{fin}^{(q,p)}\cap\mathcal{C}(\mathbb{R}^d)$ is dense in $\mathcal{H}^{(q,p)}$ with respect to the quasi-norm $\left\|\cdot\right\|_{\mathcal{H}^{(q,p)}}$. This proves the first assertion.
 
For the converse, let $T\in\left(\mathcal{H}^{(q,p)}\right)^{\ast}$. Consider $1<s<+\infty$. According to the first case, there exists a function $g\in L_{\mathrm{loc}}^{s'}$ such that
\begin{eqnarray}
T(f)=\int_{\mathbb{R}^d}f(x)g(x)dx, \label{Campanato16} 
\end{eqnarray}
for all $f\in\mathcal{H}_{fin}^{(q,p)}$, where $\mathcal{H}_{fin}^{(q,p)}$ consists of finite linear combinations of $(q,s,\delta)$-atoms, $g\in\mathcal{L}_{s',\phi_{1},\delta}$ and $\left\|g\right\|_{\mathcal{L}_{s',\phi_{1},\delta}}\lsim\left\|T\right\|$. Hence  
\begin{eqnarray*}
T(f)=\int_{\mathbb{R}^d}f(x)g(x)dx,
\end{eqnarray*}
for all $f\in\mathcal{H}_{fin}^{(q,p)}$, where $\mathcal{H}_{fin}^{(q,p)}$ consists of finite linear combinations of $(q,\infty,\delta)$-atoms, by (\ref{Campanato16}) and the fact that $(q,\infty,\delta)$-atoms are $(q,s,\delta)$-atoms. Also, $\left\|g\right\|_{\mathcal{L}_{1,\phi_{1},\delta}}\lsim\left\|T\right\|$, since $\left\|g\right\|_{\mathcal{L}_{1,\phi_{1},\delta}}\leq\left\|g\right\|_{\mathcal{L}_{s',\phi_{1},\delta}}$. This ends the proof in Case $r=+\infty$ and hence of Theorem \ref{theoremdual}. 
\end{proof}

We mention that we are not able at the moment to characterize the dual of $\mathcal{H}^{(q,p)}$ when $0<q\leq 1<p<+\infty$. 

\section{Boundedness of some classical linear operators}
 
In this section, unless otherwise specified, we assume that $0<q\leq 1$ and $q\leq p<+\infty$. 

\subsection{Calder\'on-Zygmund operator} 
 
Let $\triangle:=\left\{(x,x):\ x\in\mathbb{R}^d\right\}$ be the diagonal of $\mathbb{R}^d\times\mathbb{R}^d$.
We say that a function $K:\mathbb{R}^d\times\mathbb{R}^d\backslash\triangle\rightarrow\mathbb{C}$ is a standard kernel if there exist a constant $A>0$ and an exponent $\mu>0$ such that:
\begin{eqnarray}
|K(x,y)|\leq A|x-y|^{-d}, \label{integralsing1}
\end{eqnarray}
\begin{equation}
|K(x,y)-K(x,z)|\leq A\frac{|y-z|^{\mu}}{|x-y|^{d+\mu}}\ ,\ \text{ if }\ \ |x-y|\geq 2|y-z| \label{integralsing2}
\end{equation}
and
\begin{equation}
|K(x,y)-K(w,y)|\leq A\frac{|x-w|^{\mu}}{|x-y|^{d+\mu}}\ ,\ \text{ if }\ \ |x-y|\geq 2|x-w|. \label{integralsing3}
\end{equation}

We denote by $\mathcal{SK}(\mu,A)$ the class of all standard kernels $K$ associated with $\mu$ and $A$. 

A classical example of standard kernel is the function  $K$ defined on $\mathbb{R}^d\times\mathbb{R}^d\backslash\triangle$ by 
\begin{eqnarray}
K(x,y)=k(x-y) ,\label{inésing}
\end{eqnarray}
where $k$ is a $\mathcal{C}^{\infty}$-function on $\mathbb{R}^d\backslash\left\{0\right\}$ such that $$A_m:=\sup_{x\in\mathbb{R}^d\backslash{\left\{0\right\}}}|x|^{d+m}|\nabla^m k(x)|<+\infty,\ \forall\ m\in\mathbb{N}\cup\left\{0\right\},$$ 
and $|\nabla^m k(x)|:=\left(\underset{|\beta|=m}\sum|({\partial}^{\beta}k)(x)|^2\right)^{\frac{1}{2}}$.


\begin{defn} [\cite{JD}, Definition 5.11] An operator $T$ is a (generalized) Calder\'on-Zygmund operator if 
\begin{enumerate}
\item $T$ is bounded on $L^2$;
\item There exists a standard kernel $K$ such that for $f\in L^2$ with compact support,
\begin{eqnarray}
T(f)(x)=\int_{\mathbb{R}^d}K(x,y)f(y)dy,\ x\notin\text{supp}(f). \label{egopcalzyg1}
\end{eqnarray}
\end{enumerate}
\end{defn}

It is well known that the Calder\'on-Zygmund operator $T$ is bounded on $L^r$, $1<r<+\infty$ (see \cite{JD}, Theorem 5.10 and \cite{LG}, Theorem 8.2.1). Our second main result is the following.

\begin{thm} \label{theoremsing1}
Let $T$ be a Calder\'on-Zygmund operator with kernel $K\in \mathcal{SK}(\mu,A)$. If $\frac{d}{d+\mu}<q\leq 1$, then $T$ extends to a bounded operator from $\mathcal{H}^{(q,p)}$ to $(L^q,\ell^p)$. 
\end{thm}
\begin{proof}
Let $r>\max\left\{2;p\right\}$ be a real and $\delta\geq\left\lfloor d\left(\frac{1}{q}-1\right)\right\rfloor$ be an integer. Let $\mathcal H^{(q,p)}_{fin}$ be the space of finite linear combinations of $(q,r,\delta)$-atoms, and  
$f$ an element of this space. 
There exist a finite sequence $\left\{(\textbf{a}_n, Q^n)\right\}_{n=0}^j$ in $\mathcal{A}(q,r,\delta)$ and a finite sequence of scalars $\left\{{\lambda}_n\right\}_{n=0}^j$ such $f=\sum_{n=0}^j\lambda_n\textbf{a}_n$. Set $\widetilde{Q^n}:=2\sqrt{d}Q^n$, $n\in\left\{0,1,\ldots,j\right\}$, and denote by $x_n$ and $\ell_n$ respectively the center and side-length of $Q^n$. We have 
\begin{eqnarray*}
|T(f)(x)|\leq\sum_{n=0}^j|\lambda_n|\left(|T(\textbf{a}_n)(x)\chi_{_{\widetilde{Q^n}}}(x)|+|T(\textbf{a}_n)(x)\chi_{_{\mathbb{R}^d\backslash{\widetilde{Q^n}}}}(x)|\right),\ x\in\mathbb R^d.
\end{eqnarray*}
 

For $x\notin\widetilde{Q^n}$, we have 
\begin{eqnarray*}
|T(\textbf{a}_n)(x)|=\left|\int_{Q^n}K(x,y)\textbf{a}_n(y)dy\right|&=&\left|\int_{Q^n}[K(x,y)-K(x,x_n)]\textbf{a}_n(y)dy\right|\\
&\leq&\int_{Q^n}|K(x,x_n)-K(x,y)||\textbf{a}_n(y)|dy,
\end{eqnarray*}
by the vanishing condition of the atom $\textbf{a}_n$. Since, $|x-x_n|>2|y-x_n|$, for all $y\in Q^n$, it comes from (\ref{integralsing2}) that
\begin{eqnarray*}
|T(\textbf{a}_n)(x)|&\leq&A\int_{Q^n}\frac{|y-x_n|^{\mu}}{|x-x_n|^{d+\mu}}|\textbf{a}_n(y)|dy\\
&\lsim&A\frac{\ell_n^{\mu}}{|x-x_n|^{d+\mu}}\int_{Q^n}|\textbf{a}_n(y)|dy\lsim \frac{A}{\left\|\chi_{_{Q^n}}\right\|_q}\frac{\ell_n^{d+\mu}}{|x-x_n|^{d+\mu}}\cdot
\end{eqnarray*}
But, 
\begin{eqnarray}
\frac{\ell_n^{d+\mu}}{|x-x_n|^{d+\mu}}\lsim\left[\mathfrak{M}(\chi_{_{Q^n}})(x)\right]^{\frac{d+\mu}{d}}. \label{rectif3}
\end{eqnarray}
Hence 
\begin{equation}
|T(\textbf{a}_n)(x)|\lsim\frac{\left[\mathfrak{M}(\chi_{_{Q^n}})(x)\right]^{\frac{d+\mu}{d}}}{\left\|\chi_{_{Q^n}}\right\|_q},\ x\notin\widetilde{Q^n}.\label{applicattheo6}
\end{equation} 
Therefore,
\begin{equation*}
|T(f)(x)|\lsim\sum_{n=0}^j|\lambda_n|\left(|T(\textbf{a}_n)(x)|\chi_{_{\widetilde{Q^n}}}(x)+\frac{\left[\mathfrak{M}(\chi_{_{Q^{n}}})(x)\right]^{\frac{d+\mu}{d}}}{\left\|\chi_{_{Q^{n}}}\right\|_q}\right),
\end{equation*}
for all $x\in\mathbb{R}^d$, so that 
\begin{eqnarray*}
\left\|T(f)\right\|_{q,p}&\lsim&\left\|\sum_{n=0}^j|\lambda_n|\left(|T(\textbf{a}_n)|\chi_{_{\widetilde{Q^n}}}+\frac{\left[\mathfrak{M}(\chi_{_{Q^{n}}})\right]^{\frac{d+\mu}{d}}}{\left\|\chi_{_{Q^{n}}}\right\|_q}\right)\right\|_{q,p}\\
&\lsim&\left\|\sum_{n=0}^j|\lambda_n||T(\textbf{a}_n)|\chi_{_{\widetilde{Q^n}}}\right\|_{q,p}+\left\|\sum_{n=0}^j|\lambda_n|\frac{\left[\mathfrak{M}(\chi_{_{Q^{n}}})\right]^{\frac{d+\mu}{d}}}{\left\|\chi_{_{Q^{n}}}\right\|_q}\right\|_{q,p}.
\end{eqnarray*}
Set 
\begin{eqnarray*}
I=\left\|\sum_{n=0}^j|\lambda_n||T(\textbf{a}_n)|\chi_{_{\widetilde{Q^n}}}\right\|_{q,p}\ \text{ and }\ J=\left\|\sum_{n=0}^j|\lambda_n|\frac{\left[\mathfrak{M}(\chi_{_{Q^{n}}})\right]^{\frac{d+\mu}{d}}}{\left\|\chi_{_{Q^{n}}}\right\|_q}\right\|_{q,p}.
\end{eqnarray*}
 Let us fix $0<\eta<q$. Proceeding as in the proof of Theorem 4.6 in \cite{AbFt}, we obtain 
\begin{equation*}
J\lsim\left\|\sum_{n=0}^j\left(\frac{|\lambda_n|}{\left\|\chi_{_{Q^n}}\right\|_{q}}\right)^{\eta}\chi_{_{Q^{n}}}\right\|_{\frac{q}{\eta},\frac{p}{\eta}}^{\frac{1}{\eta}}
\end{equation*}
and
\begin{equation*}
I\leq\left\|\sum_{n=0}^j|\lambda_n|^{\eta}\left(|T(\textbf{a}_n)|\chi_{_{\widetilde{Q^n}}}\right)^{\eta}\right\|_{\frac{q}{\eta},\frac{p}{\eta}}^{\frac{1}{\eta}}\lsim\left\|\sum_{n=0}^j\left(\frac{|\lambda_n|}{\left\|\chi_{_{Q^n}}\right\|_{q}}\right)^{\eta}\chi_{_{Q^{n}}}\right\|_{\frac{q}{\eta},\frac{p}{\eta}}^{\frac{1}{\eta}},
\end{equation*}
since $1<\frac{q}{\eta}\leq\frac{p}{\eta}<\frac{r}{\eta}$ , $\text{supp}\left(|T(\textbf{a}_n)|\chi_{_{\widetilde{Q^n}}}\right)^{\eta}\subset\widetilde{Q^n}$ and
\begin{equation*}
\left\|\left(|T(\textbf{a}_n)|\chi_{_{\widetilde{Q^n}}}\right)^{\eta}\right\|_{\frac{r}{\eta}}=\left\|T(\textbf{a}_n)\chi_{_{\widetilde{Q^n}}}\right\|_{r}^{\eta}\leq\left\|T(\textbf{a}_n)\right\|_{r}^{\eta}
\lsim\left\|\textbf{a}_n\right\|_{r}^{\eta}\lsim|\widetilde{Q^n}|^{\frac{1}{\frac{r}{\eta}}-\frac{1}{\frac{q}{\eta}}}.
\end{equation*}
Hence 
\begin{equation*}
\left\|T(f)\right\|_{q,p}\lsim I+J\lsim\left\|\sum_{n=0}^j\left(\frac{|\lambda_n|}{\left\|\chi_{_{Q^n}}\right\|_{q}}\right)^{\eta}\chi_{_{Q^{n}}}\right\|_{\frac{q}{\eta},\frac{p}{\eta}}^{\frac{1}{\eta}}.
\end{equation*}
It follows that  
\begin{equation*}
\left\|T(f)\right\|_{q,p}\lsim\left\|f\right\|_{\mathcal{H}_{fin}^{(q,p)}}\lsim\left\|f\right\|_{\mathcal{H}^{(q,p)}},
\end{equation*}
By Theorem \ref{thafondamfini}. 
Therefore, $T$ is bounded from $\mathcal{H}_{fin}^{(q,p)}$ to $(L^q,\ell^p)$ and the density of $\mathcal{H}_{fin}^{(q,p)}$ in $\mathcal{H}^{(q,p)}$ with respect to the quasi-norm $\left\|\cdot\right\|_{\mathcal{H}^{(q,p)}}$ yields the result.
\end{proof}

If $T$ is a Calder\'on-Zymund operator with a sufficiently smooth standard kernel $K$, then the previous result (Theorem \ref{theoremsing1}) can be improved in the sense of values of $q$. More precisely, we have the following. 

\begin{thm} \label{theoremsing0}
Let $T$ be a Calder\'on-Zygmund operator with kernel $K\in \mathcal{SK}(\mu,A)$. Suppose that there exist an integer $m>0$ and a constant $C_m>0$ such that  
\begin{eqnarray}
|\partial_y^{\beta}K(x,y)|\leq\frac{C_m}{|x-y|^{d+|\beta|}}\ , \label{inegsing0}
\end{eqnarray}
for all multi-indexes $\beta$ with $|\beta|\leq m$ and all $(x,y)\in\mathbb{R}^d\times\mathbb{R}^d\backslash\triangle$. Then, for $\frac{d}{d+m}<q\leq 1$, $T$ extends to a bounded operator from $\mathcal{H}^{(q,p)}$ to $(L^q,\ell^p)$. 
\end{thm}
\begin{proof}
It is similar to the one of Theorem \ref{theoremsing1}. The only difference is that we use Taylor's formula to obtain Estimate (\ref{applicattheo6}), with $\mu$ replaced by $m$. More precisely, let 
$r>\max\left\{2;p\right\}$ be a real and $\delta\geq m-1$ be an integer and $\textbf{a}$ a $(q,r,\delta)$-atom supported in $Q(x_0,\ell)$. 
For $x\notin\widetilde{Q}=2\sqrt{d}Q$, we have 
\begin{eqnarray*}
|T(\textbf{a})(x)|&=&\left|\int_{Q}\left(K(x,y)-\sum_{|\beta|\leq m-1}\frac{(y-x_0)^{\beta}}{\beta!}\partial_y^{\beta}K(x,x_0)\right)\textbf{a}(y)dy\right|\\
&=&\left|\int_{Q}R(y,x_0)\textbf{a}(y)dy\right|\leq\int_{Q}|R(y,x_0)||a(y)|dy,
\end{eqnarray*}
with 
\begin{eqnarray*}
R(y,x_0)=m\sum_{|\beta|=m}\frac{(y-x_0)^{\beta}}{\beta!}\int_{0}^1(1-\theta)^{m-1}\partial_y^{\beta}K(x,\theta y+(1-\theta)x_0)d\theta,
\end{eqnarray*} 
according to Taylor's formula and the vanishing condition of the atom $\textbf{a}$. But then, 
\begin{eqnarray}
|R(y,x_0)|\lsim\frac{|y-x_0|^{m}}{|x-x_0|^{d+m}},\ y\in Q,\label{inegsing2}
\end{eqnarray}
 so that 
\begin{eqnarray*}
|T(\textbf{a})(x)|&\lsim&\int_{Q}\frac{|y-x_0|^{m}}{|x-x_0|^{d+m}}|\textbf{a}(y)|dy\\
&\lsim&\frac{\ell^{m}}{|x-x_0|^{d+m}}\int_{Q}|\textbf{a}(y)|dy\lsim \frac{\left[\mathfrak{M}(\chi_{_{Q}})(x)\right]^{\frac{d+m}{d}}}{\left\|\chi_{_{Q}}\right\|_q}.
\end{eqnarray*} 
\end{proof}

In Theorem \ref{theoremsing0}, we do not yet know whether (\ref{inegsing0}) can be replaced by the following condition: 
\begin{eqnarray}
|\partial_x^{\beta}K(x,y)|\leq\frac{C_m}{|x-y|^{d+|\beta|}}\ , \label{bisinegsing0}
\end{eqnarray}
for all multi-indexes $\beta$ with $|\beta|\leq m$ and all $(x,y)\in\mathbb{R}^d\times\mathbb{R}^d\backslash\triangle$. 

A consequence of Theorem \ref{theoremsing0} is the following result.

\begin{cor}\label{retoursur2}
If a Calder\'on-Zygmund operator $T$ satisfies the condition (\ref{inegsing0}) for any integer $m>0$, then $T$ extends to a bounded operator from $\mathcal{H}^{(q,p)}$ to $(L^q,\ell^p)$, for all $0<q\leq 1$. 
\end{cor}
\begin{proof}
For any $0<q\leq1$, by taking an integer $m>d\left(\frac{1}{q}-1\right)$, the result immediately follows from Theorem \ref{theoremsing0}.
\end{proof}


We can see that Proposition 5.4 in \cite{AbFt} is a consequence of Corollary \ref{retoursur2}. In fact, 
it's easy to verify that for any integer $m>0$, the kernel $K$ of the above mentioned proposition, satisfies  
\begin{eqnarray}
|\partial_y^{\beta}K(x,y)|\leq\frac{C(m)}{|x-y|^{d+|\beta|}}\ , \label{0applicainequal200}
\end{eqnarray}
for all multi-indexes $\beta$ with $|\beta|\leq m$ and all $(x,y)\in\mathbb{R}^d\times\mathbb{R}^d\backslash\triangle$, where $C(m)>0$ is a constant independent of $\beta$ and $(x,y)$.

In the next lemma, we give some sufficient smoothness conditions for the kernel $K$ of a Calder\'on-Zygmund operator $T$, under which this operator 
can be extended to a bounded operator from $\mathcal{H}^{(q,p)}$ to $\mathcal{H}^{(q,p)}$. 

\begin{lem}\label{propretoursur}
Let $T$ be a Calder\'on-Zygmund operator with the following properties:
\begin{enumerate}
\item There exist an integer $\delta\geq0$ and a constant $C_\delta>0$ such that  
\begin{eqnarray}
|\partial_y^{\beta}K(x,y)|\leq\frac{C_\delta}{|x-y|^{d+|\beta|}}\ , \label{propretoursur1}
\end{eqnarray}
for all $\beta$ with $|\beta|\leq d+2\delta+3$ and all $(x,y)\in\mathbb{R}^d\times\mathbb{R}^d\backslash\triangle$.
\item If $f\in L_{\mathrm{comp}}^2(\mathbb{R}^d)$ and $\int_{\mathbb{R}^d}x^{\beta}f(x)dx=0$, for all multi-indexes $\beta$ with $|\beta|\leq d+2\delta+2$, then
\begin{eqnarray}
\int_{\mathbb{R}^d}x^{\beta}T(f)(x)dx=0, \label{propretoursur2}
\end{eqnarray}
for all multi-indexes $\beta$ with $|\beta|\leq\delta$.
\end{enumerate}
If $\left\lfloor d\left(\frac{1}{q}-1\right)\right\rfloor\leq\delta$, then, for any $(a,Q)\in\mathcal{A}(q,\infty,d+2\delta+2)$ and any $1<r<+\infty$, $$\left(\frac{1}{c_1}T(a),Q\right)\in\mathcal{M}\ell(q,r,\delta),$$ where $c_1>0$ is a constant independent of the atom $a$. 
\end{lem}
\begin{proof}
Suppose that $\left\lfloor d\left(\frac{1}{q}-1\right)\right\rfloor\leq\delta$. Let $(a,Q)\in\mathcal{A}(q,\infty,d+2\delta+2)$ and $1<r<+\infty$. Set $\widetilde{Q}:=2\sqrt{d}Q$ and denote by $x_Q$ the center of $Q$. Since $\mathcal{A}(q,\infty,d+2\delta+2)\subset\mathcal{A}(q,r,d+2\delta+2)$, we have 
\begin{eqnarray}
\left\|T(a)\chi_{_{\widetilde{Q}}}\right\|_r\leq\left\|T(a)\right\|_r\lsim\left\|a\right\|_r\lsim|Q|^{\frac{1}{r}-\frac{1}{q}}.  \label{propretoursur3}
\end{eqnarray}
Let us estimate $|T(a)(x)|$, for $x\notin\widetilde{Q}$. Fix $x\notin\widetilde{Q}$. Since $a\in L^2$ and $x\notin Q$, we have, by Taylor's formula and the vanishing condition of the atom $a$, 
\begin{eqnarray*}
|T(a)(x)|=\left|\int_{Q}K(x,y)a(y)dy\right|\leq\int_{Q}|R(y,x_Q)||a(y)|dy,
\end{eqnarray*}
with
\begin{eqnarray*}
R(y,x_Q)=\delta'\sum_{|\beta|=\delta'}\frac{(y-x_Q)^{\beta}}{\beta!}\int_{0}^1(1-\theta)^{d+2\delta+2}\partial_y^{\beta}K(x,\theta y+(1-\theta)x_Q)d\theta,
\end{eqnarray*} 
and $\delta'=d+2\delta+3$. Moreover, since 
\begin{eqnarray*}
|R(y,x_Q)|\lsim\frac{|y-x_Q|^{d+2\delta+3}}{|x-x_Q|^{2d+2\delta+3}},\  y\in Q,
\end{eqnarray*}
we have 
\begin{eqnarray*}
|T(a)(x)|&\lsim&\frac{\ell(Q)^{d+2\delta+3}}{|x-x_Q|^{2d+2\delta+3}}\left\|a\right\|_{\infty}\left\|\chi_{_{Q}}\right\|_1\\
&\lsim&|Q|^{-\frac{1}{q}}\left(\frac{\ell(Q)}{|x-x_Q|}\right)^{2d+2\delta+3}\\
&\lsim&|Q|^{-\frac{1}{q}}\left(\frac{\ell(Q)+|x-x_Q|}{\ell(Q)}\right)^{-2d-2\delta-3},
\end{eqnarray*}
since $|x-x_Q|\approx\ell(Q)+|x-x_Q|$. Thus,
\begin{eqnarray}
|T(a)(x)|\lsim|Q|^{-\frac{1}{q}}\left(1+\frac{|x-x_Q|}{\ell(Q)}\right)^{-2d-2\delta-3}, \label{propretoursur5}
\end{eqnarray}
for all $x\notin\widetilde{Q}$. Therefore, there exists a positive constant $c_1$ independent of the atom $a$ such that
\begin{eqnarray*}
\left\|\frac{1}{c_1}T(a)(x)\chi_{_{\widetilde{Q}}}\right\|_r\leq|Q|^{\frac{1}{r}-\frac{1}{q}}
\end{eqnarray*}
and 
\begin{eqnarray*}
\left|\frac{1}{c_1}T(a)(x)\right|\leq|Q|^{-\frac{1}{q}}\left(1+\frac{|x-x_Q|}{\ell(Q)}\right)^{-2d-2\delta-3},
\end{eqnarray*}
for all $x\notin\widetilde{Q}$, by (\ref{propretoursur3}) and (\ref{propretoursur5}). Also, it's clear that 
\begin{eqnarray*}
\int_{\mathbb{R}^d}x^{\beta}T(a)(x)dx=0, 
\end{eqnarray*}
for all multi-indexes $\beta$ with $|\beta|\leq\delta$, by (\ref{propretoursur2}). Hence $\frac{1}{c_1}T(a)$ is a $(q,r,\delta)$-molecule centered at $Q$.
\end{proof}

\begin{remark}\label{remarqretoursur}
Lemma \ref{propretoursur} also holds, for any $(a,Q)\in\mathcal{A}(q,s,d+2\delta+2)$ and any $1<r<\infty$, provided that $s\geq\max\left\{2;r\right\}$.
\end{remark}

\begin{thm}\label{theorretoursur}
Let $T$ be a Calder\'on-Zygmund operator satisfying (\ref{propretoursur1}) and (\ref{propretoursur2}). If $\frac{d}{d+\delta+1}<q\leq 1$, then $T$ extends to a bounded operator from $\mathcal{H}^{(q,p)}$ to $\mathcal{H}^{(q,p)}$. 
\end{thm}
\begin{proof}
Noticing that $\frac{d}{d+\delta+1}<q\leq 1$ implies $\left\lfloor d\left(\frac{1}{q}-1\right)\right\rfloor\leq\delta$, we consider the space $\mathcal{H}_{fin}^{(q,p)}$ of finite linear combination of $(q,\infty,d+2\delta+2)$-atoms, and 
$f\in\mathcal{H}_{fin}^{(q,p)}\cap\mathcal{C}(\mathbb{R}^d)$. Then, there exist a finite sequence $\left\{(\textbf{a}_n, Q^n)\right\}_{n=0}^j$ of elements of $\mathcal{A}(q,\infty,d+2\delta+2)$ and a finite sequence of scalars $\left\{{\lambda}_n\right\}_{n=0}^j$ such that $f=\sum_{n=0}^j\lambda_n\textbf{a}_n$. Let $\max\left\{1;p\right\}<r<+\infty$. By Lemma \ref{propretoursur}, $\frac{1}{c_1}T(\textbf{a}_n)$ is a $(q,r,\delta)$-molecule centered at $Q^n$, for every $n\in\left\{0,1,\ldots,j\right\}$. Let us fix $0<\eta<q$. Clearly, we have  
\begin{eqnarray*}
\left\|\sum_{n=0}^j\left(\frac{|c_1\lambda_n|}{\left\|\chi_{_{Q^n}}\right\|_{q}}\right)^{\eta}\chi_{_{Q^{n}}}\right\|_{\frac{q}{\eta},\frac{p}{\eta}}^{\frac{1}{\eta}}<+\infty.
\end{eqnarray*}
Hence
\begin{eqnarray*}
T(f)=\sum_{n=0}^j\lambda_n T(\textbf{a}_n)=\sum_{n=0}^j (c_1\lambda_n)\left(\frac{1}{c_1}T(\textbf{a}_n)\right)\in\mathcal{H}^{(q,p)}
\end{eqnarray*}
and 
\begin{eqnarray*}
\left\|T(f)\right\|_{\mathcal{H}^{(q,p)}}\lsim\left\|\sum_{n=0}^j\left(\frac{|c_1\lambda_n|}{\left\|\chi_{_{Q^n}}\right\|_{q}}\right)^{\eta}\chi_{_{Q^{n}}}\right\|_{\frac{q}{\eta},\frac{p}{\eta}}^{\frac{1}{\eta}}=C\left\|\sum_{n=0}^j\left(\frac{|\lambda_n|}{\left\|\chi_{_{Q^n}}\right\|_{q}}\right)^{\eta}\chi_{_{Q^{n}}}\right\|_{\frac{q}{\eta},\frac{p}{\eta}}^{\frac{1}{\eta}},
\end{eqnarray*}
by Theorem \ref{thafondammolec2}. It follows that 
\begin{eqnarray*}
\left\|T(f)\right\|_{\mathcal{H}^{(q,p)}}\lsim\left\|f\right\|_{\mathcal{H}_{fin}^{(q,p)}}\lsim\left\|f\right\|_{\mathcal{H}^{(q,p)}},
\end{eqnarray*}
by Theorem \ref{thafondamfini}. Therefore, $T$ is bounded from $\mathcal{H}_{fin}^{(q,p)}\cap\mathcal{C}(\mathbb{R}^d)$ to $\mathcal{H}^{(q,p)}$ and the density of $\mathcal{H}_{fin}^{(q,p)}\cap\mathcal{C}(\mathbb{R}^d)$ in $\mathcal{H}^{(q,p)}$ with respect to the quasi-norm $\left\|\cdot\right\|_{\mathcal{H}^{(q,p)}}$ yields the result.
\end{proof}

By Remark \ref{remarqretoursur}, in the previous proof, we can also use the set $\mathcal{A}(q,r,d+2\delta+2)$, with $\max\left\{2;p\right\}<r<+\infty$, instead of $\mathcal{A}(q,\infty,d+2\delta+2)$. 

\begin{cor}
Let $T$ be a Calder\'on-Zygmund operator which satisfies (\ref{propretoursur1}) and (\ref{propretoursur2}), for any integer $\delta\geq0$. Then, for all $0<q\leq1$, $T$ extends to a bounded operator from $\mathcal{H}^{(q,p)}$ to $\mathcal{H}^{(q,p)}$.
\end{cor}

\subsection{Convolution Operator}
Given a function $k$ defined and locally integrable on $\mathbb{R}^d\backslash\left\{0\right\}$, we say that a tempered distribution $K$ in $\mathbb{R}^d$ ($K\in\mathcal{S'}:=\mathcal{S'}(\mathbb{R}^d)$) coincides with the function $k$ on $\mathbb{R}^d\backslash\left\{0\right\}$, if 
\begin{eqnarray}
\left\langle K, \psi\right\rangle=\int_{\mathbb{R}^d}k(x)\psi(x)dx,\label{applicattheo090}
\end{eqnarray}
for all $\psi\in\mathcal{S}$, with $\text{supp}(\psi)\subset\mathbb{R}^d\backslash\left\{0\right\}$. It's clear that (\ref{applicattheo090}) is not valid for any $\psi\in\mathcal{S}$. Here, we are interested in tempered distributions $K$ in $\mathbb{R}^d$ that coincide with a function $k$ on $\mathbb{R}^d\backslash\left\{0\right\}$ and that have the form 
\begin{eqnarray}
\left\langle K, \psi\right\rangle=\lim_{j\rightarrow+\infty}\int_{|x|\geq\sigma_j}k(x)\psi(x)dx,\ \ \ \psi\in\mathcal{S},  \label{applicattheo091}
\end{eqnarray}
for some sequence $\sigma_j\downarrow 0$ as $j\rightarrow+\infty$ and independent of $\psi$. Also, we consider convolution operators $T$: $T(f)=K\ast f$, $f\in\mathcal{S}$. Thus, when $\widehat{K}\in L^{\infty}$, we have 
\begin{eqnarray}
T(f)(x)=\int_{\mathbb{R}^d}k(x-y)f(y)dy, \label{applicattheo093}
\end{eqnarray}
for all $f\in L^2$ with compact support and all $x\notin\text{supp}(f)$. For (\ref{applicattheo093}), see \cite{MA}, Chap. 3, section 3, p. 113. From now on, unless otherwise specified, the letter $K$ stands both for the tempered distribution $K$ and the associated function $k$. We recall the following theorem.

\begin{thm}[\cite{JD}, Theorem 5.1] \label{theoremsingaj}  Let $K$ be a tempered distribution in $\mathbb{R}^d$ which coincides with a locally integrable function on $\mathbb{R}^d\backslash\left\{0\right\}$ and is such that 
\begin{eqnarray}
|\widehat{K}(\xi)|\leq A, \label{applicattheo9}
\end{eqnarray} 
\begin{eqnarray}
\int_{|x|>2|y|}|K(x-y)-K(x)|dx\leq B,\ y\in\mathbb{R}^d. \label{applicattheo10}
\end{eqnarray}
then, for $1<r<+\infty$,
\begin{eqnarray*}
\left\|K\ast f\right\|_r\leq C_r\left\|f\right\|_r
\end{eqnarray*} 
and
\begin{eqnarray*}
\left|\left\{x\in\mathbb{R}^d:\ |K\ast f(x)|>\lambda\right\}\right|\leq\frac{C}{\lambda}\left\|f\right\|_1.
\end{eqnarray*} 
\end{thm}

Notice that $C_r:=C(d,r,A,B)$ (see \cite{JD}, p. 110) and (\ref{applicattheo9}) is understood as $|\widehat{K}(\xi)|\leq A$ for almost all $\xi\in\mathbb{R}^d$. For our results, we will need conditions stronger than (\ref{applicattheo10}) which are recalled in the following two remarks.

\begin{remark}[\cite{JD}, Proposition 5.2] \label{remarqsingaj}  The H\"ormander condition (\ref{applicattheo10}) holds, if for every $x\neq0$
\begin{eqnarray}
|\nabla K(x)|\leq\frac{C}{|x|^{d+1}}\ , \label{applicattheo11}
\end{eqnarray}
where $C>0$ is a constant independent of $x$. 
\end{remark}

\begin{remark}\label{remarqsingaj1} 
If there exist an integer $\delta>0$ and a constant $C_\delta>0$ such that 
\begin{eqnarray}
|\partial^{\beta}K(x)|\leq\frac{C_\delta}{|x|^{d+|\beta|}}\ , \label{applicattheo101}
\end{eqnarray}
for all multi-indexes $\beta$ with $|\beta|\leq\delta$ and all $x\neq0$, then (\ref{applicattheo11}) holds. Therefore, (\ref{applicattheo10}) is also satisfied. 
\end{remark}

This follows from the definition of $|\nabla K(x)|^2$.

\begin{thm} \label{theoremsing3}
Let $K$ be a tempered distribution in $\mathbb{R}^d$ which coincides wich a locally integrable function on $\mathbb{R}^d\backslash\left\{0\right\}$ and is such that 
\begin{eqnarray*}
|\widehat{K}(\xi)|\leq A,
\end{eqnarray*} 
\begin{eqnarray}
|\nabla K(x)|\leq\frac{B}{|x|^{d+1}}\ , \label{applicattheo12}
\end{eqnarray}
for all $x\neq0$. If $\frac{d}{d+1}<q\leq1$, then the operator $T(f)=K\ast f$, for all $f\in\mathcal{S}$, extends to a bounded operator from $\mathcal{H}^{(q,p)}$ to $(L^q,\ell^p)$.
\end{thm}
\begin{proof}
Let $\delta\geq\left\lfloor d\left(\frac{1}{q}-1\right)\right\rfloor$ be an integer, $r>\max\left\{2;p\right\}$ be a real and $\mathcal{H}_{fin}^{(q,p)}$ be the space of finite linear combination of $(q,r,\delta)$-atoms. Fix $f\in\mathcal{H}_{fin}^{(q,p)}$. Then, there exist a finite sequence $\left\{(\textbf{a}_n, Q^n)\right\}_{n=0}^j$ of elements of $\mathcal{A}(q,r,\delta)$ and a finite sequence of scalars $\left\{{\lambda}_n\right\}_{n=0}^j$ such that $f=\sum_{n=0}^j\lambda_n\textbf{a}_n$. Moreover, $T$ is bounded on $L^r$, by Remark \ref{remarqsingaj} and Theorem \ref{theoremsingaj}. Set $\widetilde{Q^n}:=2\sqrt{d}Q^n$, $n\in\left\{0,1,\ldots,j\right\}$, and denote by $x_n$ and $\ell_n$ respectively the center and side-length of $Q^n$. We have 
\begin{eqnarray*}
\left\|T(f)\right\|_{q,p}\lsim\left\|\sum_{n=0}^j|\lambda_n||T(\textbf{a}_n)|\chi_{_{\widetilde{Q^n}}}\right\|_{q,p}+\left\|\sum_{n=0}^j|\lambda_n||T(\textbf{a}_n)|\chi_{_{\mathbb{R}^d\backslash{\widetilde{Q^n}}}}\right\|_{q,p}.
\end{eqnarray*}
Set $$I=\left\|\sum_{n=0}^j|\lambda_n||T(\textbf{a}_n)|\chi_{_{\widetilde{Q^n}}}\right\|_{q,p}\ \text{ and }\ J=\left\|\sum_{n=0}^j|\lambda_n||T(\textbf{a}_n)|\chi_{_{\mathbb{R}^d\backslash{\widetilde{Q^n}}}}\right\|_{q,p}.$$
Let us fix $0<\eta<q$. Arguing as in the proof of Theorem \ref{theoremsing1}, we have
\begin{eqnarray*}
I\leq\left\|\sum_{n=0}^j|\lambda_n|^{\eta}\left(|T(\textbf{a}_n)|\chi_{_{\widetilde{Q^n}}}\right)^{\eta}\right\|_{\frac{q}{\eta},\frac{p}{\eta}}^{\frac{1}{\eta}}\lsim\left\|\sum_{n=0}^j\left(\frac{|\lambda_n|}{\left\|\chi_{_{Q^n}}\right\|_{q}}\right)^{\eta}\chi_{_{Q^n}}\right\|_{\frac{q}{\eta},\frac{p}{\eta}}^{\frac{1}{\eta}}.
\end{eqnarray*}
To estimate $J$, it suffices to show that 
\begin{eqnarray}
|T(\textbf{a}_{n})(x)|\lsim\frac{\left[\mathfrak{M}(\chi_{_{Q^n}})(x)\right]^{\frac{d+1}{d}}}{\left\|\chi_{_{Q^n}}\right\|_q}\ ,\label{applicattheo13}
\end{eqnarray}
for all $x\notin\widetilde{Q^n}$. 
For this, we consider $x\notin\widetilde{Q^n}$. It's clear that $x\notin Q^n$. Hence $x-y\neq0$ and $K(x-y)$ is well defined, for all $y\in Q^n$. Thus,
\begin{eqnarray*}
|T(\textbf{a}_n)(x)|&=&|K\ast\textbf{a}_n(x)|=\left|\int_{Q^n}K(x-y)\textbf{a}_n(y)dy\right|\\
&\leq&\int_{Q^n}|K(x-x_n)-K(x-y)||\textbf{a}_n(y)|dy,
\end{eqnarray*}
by the vanishing condition of the atom $\textbf{a}_n$. Furthermore, by the mean value Theorem,
\begin{eqnarray*}
|K(x-x_n)-K(x-y)|=|\nabla K(\xi)||x_n-y|,
\end{eqnarray*}
where $\xi=\theta(x-x_n)+(1-\theta)(x-y)$, for some $0<\theta<1$. Also, $|x-x_n|>2|y-x_n|$ and $|\xi|>\frac{1}{2}|x-x_n|>0$, for all $y\in Q^n$. Therefore,
\begin{eqnarray*}
|T(\textbf{a}_n)(x)|\lsim\int_{Q^n}\frac{|x_n-y|}{|x_n-x|^{d+1}}|\textbf{a}_n(y)|dy.
\end{eqnarray*}
Thus, arguing as in the proof of Theorem \ref{theoremsing1}, we obtain 
\begin{eqnarray*}
|T(\textbf{a}_{n})(x)|\lsim\frac{\left[\mathfrak{M}(\chi_{_{Q^n}})(x)\right]^{\frac{d+1}{d}}}{\left\|\chi_{_{Q^n}}\right\|_q}.
\end{eqnarray*}
This proves (\ref{applicattheo13}). 
 Proceeding as in the proof of Theorem 4.6 in \cite{AbFt}, we have   
\begin{eqnarray*}
J\lsim\left\|\sum_{n=0}^j|\lambda_n|\frac{\left[\mathfrak{M}(\chi_{_{Q^n}})\right]^{\frac{d+1}{d}}}{\left\|\chi_{_{Q^n}}\right\|_q}\right\|_{q,p}\lsim\left\|\sum_{n=0}^j\left(\frac{|\lambda_n|}{\left\|\chi_{_{Q^n}}\right\|_{q}}\right)^{\eta}\chi_{_{Q^{n}}}\right\|_{\frac{q}{\eta},\frac{p}{\eta}}^{\frac{1}{\eta}}.
\end{eqnarray*}
Hence 
\begin{eqnarray*}
\left\|T(f)\right\|_{q,p}\lsim\left\|\sum_{n=0}^j\left(\frac{|\lambda_n|}{\left\|\chi_{_{Q^n}}\right\|_{q}}\right)^{\eta}\chi_{_{Q^n}}\right\|_{\frac{q}{\eta},\frac{p}{\eta}}^{\frac{1}{\eta}}.
\end{eqnarray*}
And we finish the proof as the one of Theorem \ref{theoremsing1}.
\end{proof}

If $K$ is sufficiently smooth, the conclusion of Theorem \ref{theoremsing3} can be improved in the sense of values of $q$. More precisely, we have the following result which is proved as Theorem \ref{theoremsing3} with appropriate Taylor formula. The proof is not given.

\begin{thm} \label{theoremsing4}
Let $K$ be a tempered distribution in $\mathbb{R}^d$ that coincides with a locally integrable function on $\mathbb{R}^d\backslash\left\{0\right\}$ and is such that 
\begin{eqnarray*}
|\widehat{K}(\xi)|\leq A,
\end{eqnarray*} 
and, there exist an integer $\delta>0$ and a constant $B>0$ such that
\begin{eqnarray}
|\partial^{\beta}K(x)|\leq\frac{B}{|x|^{d+|\beta|}}\ , \label{applicattheo15}
\end{eqnarray}
for all $x\neq0$ and all multi-indexes $\beta$ with $|\beta|\leq\delta$. If $\frac{d}{d+\delta}<q\leq 1$, then the operator $T(f)=K\ast f$, for all $f\in\mathcal{S}$, extends to a bounded operator from $\mathcal{H}^{(q,p)}$ to $(L^q,\ell^p)$.
\end{thm}

We give some consequences of Theorem \ref{theoremsing4}.

\begin{cor}\label{corollsing4}
Let $K$ be a tempered distribution in $\mathbb{R}^d$ which coincides with a locally integrable function on $\mathbb{R}^d\backslash\left\{0\right\}$ and is such that 
\begin{eqnarray*}
|\widehat{K}(\xi)|\leq A,
\end{eqnarray*} 
and, for any integer $\delta>0$, satisfies (\ref{applicattheo15}). Then, for all $0<q\leq 1$, the operator $T(f)=K\ast f$, for all $f\in\mathcal{S}$, extends to a bounded operator from $\mathcal{H}^{(q,p)}$ to $(L^q,\ell^p)$.
\end{cor}

\begin{cor}\label{appliHilberttransf1}
For $d=1$, the Hilbert transform $H$ extends to a bounded operator from $\mathcal{H}^{(q,p)}$ to $(L^q,\ell^p)$, for all $0<q\leq 1$.
\end{cor}
\begin{proof}
 We recall that the Hilbert transform $H(f)$ of $f\in\mathcal{S}$ is defined by  
\begin{eqnarray}
H(f)(x)=K\ast f(x)=\lim_{\epsilon\rightarrow0}H^{\epsilon}(f)(x),\ \ x\in\mathbb{R}, \label{inegHilberttransf1}
\end{eqnarray}
where $K(x)=\frac{1}{\pi}\frac{1}{x}$ and
\begin{eqnarray*}
H^{\epsilon}(f)(x)=\frac{1}{\pi}\int_{|y|\geq\epsilon}\frac{f(x-y)}{y}dy=\frac{1}{\pi}\int_{|x-y|\geq\epsilon}\frac{f(y)}{x-y}dy. 
\end{eqnarray*}
See \cite{LGkos}, Definition 4.1.1, p. 250. Also, 
\begin{eqnarray}
\widehat{K}(\xi)=-i\ \text{sgn}(\xi), \label{inegHilberttransf5}
\end{eqnarray}
for $\xi\in\mathbb{R}$, where $\text{sgn}(\xi)=\left\{\begin{array}{lll}\frac{\xi}{|\xi|}\ ,&\text{ if }&\xi\neq 0\\
0\ ,&\text{ if }&\xi=0.\end{array}\right.$\\ For (\ref{inegHilberttransf5}), see \cite{LGkos}, (4.1.11), p. 253.  
 Hence 
$|\widehat{K}(\xi)|\leq 1,$
for all $\xi\in\mathbb{R}$, and, for any integer $\delta>0$, 
\begin{eqnarray}
|\partial^{\beta}K(x)|\leq\frac{C(\delta)}{|x|^{1+\beta}}\ , \label{inegHilberttransf6}
\end{eqnarray}
for all $x\neq0$ and all $\beta\leq\delta$. Thus, the distribution $K$ satisfies the conditions of Corollary \ref{corollsing4} and the result follows.
\end{proof}

\begin{cor}\label{appliRiesztransf1}
The Riesz transforms $R_j$, $1\leq j\leq d$, extend to bounded operators from $\mathcal{H}^{(q,p)}$ to $(L^q,\ell^p)$, for every $0<q\leq 1$.
\end{cor}
\begin{proof}
We recall that the Riesz transform $R_j(f)$, $1\leq j\leq d$, of $f\in\mathcal{S}$ is given by  
\begin{eqnarray*}
R_j(f)(x)=(K_j\ast f)(x)=\frac{\Gamma\left(\frac{d+1}{2}\right)}{\pi^{\frac{d+1}{2}}}\text{ p.v. }\int_{\mathbb{R}^d}\frac{y_j}{|y|^{d+1}}f(x-y)dy,\ \ x\in\mathbb{R}^d,
\end{eqnarray*}
where $K_j(x)=\frac{\Gamma\left(\frac{d+1}{2}\right)}{\pi^{\frac{d+1}{2}}}\frac{x_j}{|x|^{d+1}}$ and
\begin{eqnarray*}
\text{ p.v. }\int_{\mathbb{R}^d}\frac{y_j}{|y|^{d+1}}f(x-y)dy=\lim_{\epsilon\rightarrow0}\int_{|y|\geq\epsilon}\frac{y_j}{|y|^{d+1}}f(x-y)dy.
\end{eqnarray*}
See \cite{LGkos}, Definition 4.1.13, p. 259. Also 
\begin{eqnarray}
\widehat{K}_j(\xi)=-i\frac{\xi_j}{|\xi|}\ , \label{inegRiesztransf2}
\end{eqnarray}
for $\xi\neq0$. For (\ref{inegRiesztransf2}), see \cite{LGkos}, Proposition 4.1.14, p. 260 and \cite{JD}, Chap. 4, (4.8), p. 76. Hence 
\begin{eqnarray}
\left\|\widehat{K}_j\right\|_{\infty}\leq 1, \label{inegRiesztransf3}
\end{eqnarray}
and, for any integer $\delta>0$, 
\begin{eqnarray}
|\partial^{\beta}K_j(x)|\leq\frac{C(d,\delta)}{|x|^{d+|\beta|}}\ , \label{inegRiesztransf4}
\end{eqnarray}
for all $x\neq0$ and all multi-indexes $\beta$ with $|\beta|\leq\delta$. For (\ref{inegRiesztransf4}), see \cite{SZLU}, p. 120. Thus, the distribution $K_j$ satisfies the conditions of Corollary \ref{corollsing4} and the result follows.    
\end{proof}

In the next result, we prove that the convolution operator defined in Theorem \ref{theoremsing4} is bounded in Hardy-amalgam spaces.

\begin{thm} \label{theoremsing5}
Let $K$ be a tempered distribution in $\mathbb{R}^d$ that coincides with a locally integrable function on $\mathbb{R}^d\backslash\left\{0\right\}$ and is such that
\begin{eqnarray*}
|\widehat{K}(\xi)|\leq A,
\end{eqnarray*} 
and, there exist an integer $\delta>0$ and a constant $B>0$ such that
\begin{eqnarray}
|\partial^{\beta}K(x)|\leq\frac{B}{|x|^{d+|\beta|}}\ ,\label{bisapplicattheo15}
\end{eqnarray}
for all $x\neq0$ and all multi-indexes $\beta$ with $|\beta|\leq\delta$. If $\frac{d}{d+\delta}<q\leq 1$, then the operator $T(f)=K\ast f$, for all $f\in\mathcal{S}$, extends to a bounded operator from $\mathcal{H}^{(q,p)}$ to $\mathcal{H}^{(q,p)}$.  
\end{thm}
\begin{proof}
Under the assumption that $\frac{d}{d+\delta}<q\leq 1$, we have $\delta\geq\left\lfloor d\left(\frac{1}{q}-1\right)\right\rfloor+1$. Let $r>\max\left\{2;p\right\}$ be a real. We consider the space $\mathcal{H}_{fin}^{(q,p)}$ of finite linear combination of $(q,r,\delta)$-atoms. For $f\in\mathcal{H}_{fin}^{(q,p)}$, there exist a finite sequence $\left\{(\textbf{a}_n, Q^n)\right\}_{n=0}^j$ of elements of $\mathcal{A}(q,r,\delta)$ and a finite sequence of scalars $\left\{{\lambda}_n\right\}_{n=0}^j$ such that $f=\sum_{n=0}^j\lambda_n\textbf{a}_n$. Set $\widetilde{Q^n}:=2\sqrt{d}Q^n$, $n\in\left\{0,1,\ldots,j\right\}$, and denote by $x_n$ and $\ell_n$ respectively the center and side-length of $Q^n$. We have, for all $x\in\mathbb{R}^d$,
\begin{eqnarray*}
\mathcal{M}_{0}(T(f))(x)\leq\sum_{n=0}^j|\lambda_n|\left[\mathcal{M}_{0}(T(\textbf{a}_n))(x)\chi_{_{\widetilde{Q^n}}}(x)+\mathcal{M}_{0}(T(\textbf{a}_n))(x)\chi_{_{\mathbb{R}^d\backslash{\widetilde{Q^n}}}}(x)\right].
\end{eqnarray*}
Therefore, 
\begin{eqnarray*}
\left\|\mathcal{M}_{0}(T(f))\right\|_{q,p}\lsim I+J, 
\end{eqnarray*}
with $$I=\left\|\sum_{n=0}^j|\lambda_n|\mathcal{M}_{0}(T(\textbf{a}_n))\chi_{_{\widetilde{Q^n}}}\right\|_{q,p}\ \text{ and }\ J=\left\|\sum_{n=0}^j|\lambda_n|\mathcal{M}_{0}(T(\textbf{a}_n))\chi_{_{\mathbb{R}^d\backslash{\widetilde{Q^n}}}}\right\|_{q,p}.$$ Let us fix $0<\eta<q$. Since $\text{supp}\left(\mathcal{M}_{0}(T(\textbf{a}_n))\chi_{_{\widetilde{Q^n}}}\right)^{\eta}\subset\widetilde{Q^n}$ , $1<\frac{q}{\eta}\leq\frac{p}{\eta}<\frac{r}{\eta}$ and
\begin{eqnarray*}
\left\|\left(\mathcal{M}_{0}(T(\textbf{a}_n))\chi_{_{\widetilde{Q^n}}}\right)^{\eta}\right\|_{\frac{r}{\eta}}
&\leq&\left\|\mathcal{M}_{0}(T(\textbf{a}_n))\right\|_{r}^{\eta}\\
&\lsim&\left\|\mathfrak{M}(T(\textbf{a}_n))\right\|_{r}^{\eta}\lsim|\widetilde{Q^n}|^{\frac{1}{\frac{r}{\eta}}-\frac{1}{\frac{q}{\eta}}}, 
\end{eqnarray*}
we obtain, as in the proof of Theorem \ref{theoremsing3},  
\begin{eqnarray*}
I\leq\left\|\sum_{n=0}^j|\lambda_n|^{\eta}\left(\mathcal{M}_{0}(T(\textbf{a}_n))\chi_{_{\widetilde{Q^n}}}\right)^{\eta}\right\|_{\frac{q}{\eta},\frac{p}{\eta}}^{\frac{1}{\eta}}\lsim\left\|\sum_{n=0}^j\left(\frac{|\lambda_n|}{\left\|\chi_{_{Q^n}}\right\|_{q}}\right)^{\eta}\chi_{_{Q^n}}\right\|_{\frac{q}{\eta},\frac{p}{\eta}}^{\frac{1}{\eta}}.
\end{eqnarray*}
To estimate $J$, it suffices to prove that 
\begin{eqnarray}
\mathcal{M}_{0}(T(\textbf{a}_n))(x)\lsim\frac{\left[\mathfrak{M}(\chi_{_{Q^n}})(x)\right]^{\frac{d+\delta}{d}}}{\left\|\chi_{_{Q^n}}\right\|_q}\ , \label{applicattheo17}
\end{eqnarray}
for all $x\notin\widetilde{Q^n}$. 
 To show (\ref{applicattheo17}), we consider $x\notin\widetilde{Q^n}$. For every $t>0$, set $K^{(t)}=K\ast\varphi_t$. We have
\begin{eqnarray}
\sup_{t>0}|\widehat{K^{(t)}}(z)|\leq\left\|\widehat{\varphi}\right\|_{\infty}A\ , \label{applicattheo18}
\end{eqnarray}
for all $z\in\mathbb{R}^d$, and  
\begin{eqnarray}
\sup_{t>0}|\partial^{\beta}_z K^{(t)}(z)|\leq \frac{C_{\varphi,A,B,\delta,d}}{|z|^{d+|\beta|}}\ , \label{applicattheo19}
\end{eqnarray}
for all $z\neq0$ and all multi-indices $\beta$ with $|\beta|\leq\delta$. For (\ref{applicattheo18}) and (\ref{applicattheo19}), see \cite{LG}, (6.7.22) and (6.7.23), pp. 100-101 (also see \cite{MA}, Lemma, p. 117). Thus, using Taylor's formula at order $\delta-1$ and the vanishing condition of the atom $\textbf{a}_n$, we obtain 
\begin{eqnarray*}
|(T(\textbf{a}_n)\ast\varphi_t)(x)|&=&|\textbf{a}_n\ast K^{(t)}(x)|=\left|\int_{Q^n}K^{(t)}(x-y)\textbf{a}_n(y)dy\right|\\
&\leq&\int_{Q^n}|R^{(t)}(x-y,x-x_n)||\textbf{a}_n(y)|dy,
\end{eqnarray*}
with
\begin{eqnarray*}
|R^{(t)}(x-y,x-x_n)|&=&\left|\delta\sum_{|\beta|=\delta}\frac{(x_n-y)^{\beta}}{\beta!}\int_{0}^1(1-\theta)^{\delta-1}\partial^{\beta}_x K^{(t)}((x-x_n)+\theta(x_n-y))d\theta\right|\\
&\lsim&\frac{|x_n-y|^{\delta}}{|x-x_n|^{d+\delta}}\ ,
\end{eqnarray*}
for all $y\in Q^n$, by (\ref{applicattheo19}). Hence 
\begin{eqnarray*}
|(T(\textbf{a}_n)\ast\varphi_t)(x)|\lsim\int_{Q^n}\frac{|x_n-y|^{\delta}}{|x-x_n|^{d+\delta}}|\textbf{a}_n(y)|dy\lsim\frac{\left[\mathfrak{M}(\chi_{_{Q^n}})(x)\right]^{\frac{d+\delta}{d}}}{\left\|\chi_{_{Q^n}}\right\|_q},
\end{eqnarray*}
and consequently
\begin{eqnarray*}
J\lsim\left\|\sum_{n=0}^j|\lambda_n|\frac{\left[\mathfrak{M}(\chi_{_{Q^n}})\right]^{\frac{d+\delta}{d}}}{\left\|\chi_{_{Q^n}}\right\|_q}\right\|_{q,p}\lsim \left\|\sum_{n=0}^j\left(\frac{|\lambda_n|}{\left\|\chi_{_{Q^n}}\right\|_{q}}\right)^{\eta}\chi_{_{Q^{n}}}\right\|_{\frac{q}{\eta},\frac{p}{\eta}}^{\frac{1}{\eta}}.
\end{eqnarray*}
Finally, 
\begin{eqnarray*}
\left\|\mathcal{M}_{0}(T(f))\right\|_{q,p}\lsim\left\|\sum_{n=0}^j\left(\frac{|\lambda_n|}{\left\|\chi_{_{Q^n}}\right\|_{q}}\right)^{\eta}\chi_{_{Q^n}}\right\|_{\frac{q}{\eta},\frac{p}{\eta}}^{\frac{1}{\eta}}
\end{eqnarray*}
and we end the proof as the one of Theorem \ref{theoremsing1}.
\end{proof}

We give some consequences of Theorem \ref{theoremsing5}.

\begin{cor}\label{corollsing5}
Let $K$ be a tempered distribution in $\mathbb{R}^d$ which coincides with a locally integrable function on $\mathbb{R}^d\backslash\left\{0\right\}$ and is such that 
\begin{eqnarray*}
|\widehat{K}(\xi)|\leq A,
\end{eqnarray*} 
and, for any integer $\delta>0$, satisfies (\ref{bisapplicattheo15}). Then, the operator $T(f)=K\ast f$, for all $f\in\mathcal{S}$, extends to a bounded operator from $\mathcal{H}^{(q,p)}$ to $\mathcal{H}^{(q,p)}$, for all $0<q\leq 1$.
\end{cor}

\begin{cor} \label{appliRiesztransf2}
The Riesz transforms  $R_j$, $j\in\left\{1,2,\ldots,d\right\}$, extend to bounded operators on $\mathcal{H}^{(q,p)}$, for all $0<q\leq 1$. In particular, when $d=1$, the Hilbert transform $H$ extends to a bounded operator on $\mathcal{H}^{(q,p)}$, for all $0<q\leq 1$. 
\end{cor}
\begin{proof}
Corollary \ref{appliRiesztransf2} immediately follows from corollary \ref{corollsing5}.
\end{proof}

Another consequence of Corollary \ref{corollsing5} is the following.

\begin{cor}[\cite{AbFt}, Proposition 5.5] \label{propapplica1}
Let $k\in\mathcal{S}$ and $T$ be the operator defined by  
\begin{eqnarray}
T(f)=k\ast f,
\label{applicainequal2re}
\end{eqnarray}
for every $f\in L^2$. Then, $T$ extends to a bounded operator from $\mathcal{H}^{(q,p)}$ to $\mathcal{H}^{(q,p)}$, for all $0<q\leq 1$. 
\end{cor}
\begin{proof}
Since $k\in\mathcal{S}$, we have clearly $k\in\mathcal{S'}$ and $\widehat{k}\in\mathcal{S}$. Therefore, there exists a constant $A>0$ such that  
\begin{eqnarray}
|\widehat{k}(\xi)|\leq A, \label{inécorollsing50}
\end{eqnarray}
for all $\xi\in\mathbb{R}^d$. Also, for any integer $\delta>0$, it's easy to verify that there exists a constant $B>0$ such that 
\begin{eqnarray}
|\partial^{\beta}k(x)|\leq\frac{B}{|x|^{d+|\beta|}}\ , \label{inécorollsing51}
\end{eqnarray}
for all $x\neq0$ and all multi-indexes $\beta$ with $|\beta|\leq\delta$. With (\ref{inécorollsing50}) and (\ref{inécorollsing51}), the result immediately follows from Corollary \ref{corollsing5}.
\end{proof}

\subsection{Riesz Potential operators}

In this subsection, we study the boundedness of Riesz potential operators $I_\alpha$, $0<\alpha<d$, in $\mathcal{H}^{(q,p)}$ spaces. It's well known that the operator $I_\alpha$ is bounded from $L^r$ to $L^s$, for $1<r<s<+\infty$ such that $\frac{1}{s}=\frac{1}{r}-\frac{\alpha}{d}$ (see \cite{JD}, Theorem 4.18, p. 89, \cite{LG}, Theorem 6.1.3, p. 3 and \cite{EMS}, Chap. V, Theorem 1, pp. 119-120). Also, $I_\alpha$ is bounded from $\mathcal{H}^q$ to $\mathcal{H}^p$, when $\frac{1}{p}=\frac{1}{q}-\frac{\alpha}{d}$ (see \cite{SZLU}, Theorem 3.1, p. 105 and \cite{MA}, Chap. III, p. 136). We recall that  
\begin{eqnarray}
I_\alpha(\phi)(x)=\frac{1}{\gamma_\alpha}\int_{\mathbb{R}^d}\frac{\phi(y)}{|x-y|^{d-\alpha}}dy, \label{potenRiesz0}
\end{eqnarray}
for all $\phi\in\mathcal{S}$, with $\gamma_\alpha=\pi^{\frac{d}{2}-\alpha}\frac{\Gamma\left(\frac{\alpha}{2}\right)}{\Gamma\left(\frac{d-\alpha}{2}\right)}\cdot$ For (\ref{potenRiesz0}), see \cite{JD}, Chap. 4, p. 88. Denoting by $K_\alpha$ the Riesz kernel, (\ref{potenRiesz0}) writes  
\begin{eqnarray}
I_\alpha(\phi)(x)=K_\alpha\ast\phi(x), \label{potenRiesz1}
\end{eqnarray}
with $$K_\alpha(x)=\frac{1}{\gamma_\alpha}|x|^{-(d-\alpha)},\ \  x\neq 0.$$ (\ref{potenRiesz1}) is also equal in the sense of distributions to 
\begin{eqnarray}
\widehat{I_\alpha(\phi)}(\xi)=|\xi|^{-\alpha}\widehat{\phi}(\xi), \label{potenRiesz2}
\end{eqnarray}
where 
\begin{eqnarray}
|\xi|^{-\alpha}=\widehat{K_\alpha}(\xi)\ , \label{potenRiesz3}
\end{eqnarray}
for all $\xi\neq 0$. For (\ref{potenRiesz2}) and (\ref{potenRiesz3}), see respectively \cite{JD}, Chap. 4, p. 88 and (4.3), p. 71.  

For our third  main result, we will need the following lemma. 

\begin{lem} \label{propofondaRiez}
Let $0<\sigma<d$, $0<w\leq h<+\infty$ and $0<u\leq v<+\infty$ such that 
\begin{eqnarray}
\frac{1}{u}=\frac{1}{w}-\frac{\sigma}{d} \label{potenRiesz}
\end{eqnarray}
and 
\begin{eqnarray}
\frac{1}{v}=\frac{1}{h}-\frac{\sigma}{d}\cdot \label{0potenRiesz}
\end{eqnarray}
Then, for all sequences of cubes $\left\{Q^n\right\}_{n\geq 0}$ and all sequences of nonnegative scalars $\left\{\delta_n\right\}_{n\geq 0}$, we have
\begin{eqnarray}
\left\|\sum_{n\geq0}\delta_n\ell_n^\sigma\chi_{_{Q^n}}\right\|_{u,v}\leq C\left\|\sum_{n\geq0}\delta_n\chi_{_{Q^n}}\right\|_{w,h}, \label{poten0Riesz}
\end{eqnarray}
where $\ell_n:=\ell(Q^n)$ and $C:=C(d,\sigma,w,h,u,v)>0$ is a constant independent of the sequences $\left\{Q^n\right\}_{n\geq 0}$ and $\left\{\delta_n\right\}_{n\geq 0}$.
\end{lem}
\begin{proof} 
The idea of our proof is the one of \cite{SAY}, Lemma 5.2, p. 138. Let  $\left\{Q^n\right\}_{n\geq 0}$ be a sequence of cubes and $\left\{\delta_n\right\}_{n\geq 0}$ a sequence of nonnegative scalars.  

We first suppose  that each $Q^n$ is a dyadic cube, namely, $$Q^n=2^{-\upsilon_n}\prod_{i=1}^d[m_{n,i}\ ,\ m_{n,i}+1):=\prod_{i=1}^d[2^{-\upsilon_n}m_{n,i}\ ,\ 2^{-\upsilon_n}(m_{n,i}+1))\ ,$$ where $\upsilon_n\in\mathbb{Z}$ is the generation of $Q^n$ and $(m_{n,1},m_{n,2},\ldots,m_{n,d})\in\mathbb{Z}^d$. 
Since dyadic cubes form a grid, we can assume that $Q^n\neq Q^m$, for $0\leq n<m<+\infty$.

 Let 
$$F(x):=\sum_{n\geq0}\delta_n\chi_{_{Q^n}}(x)\ \text{ and }\ G(x):=\sum_{n\geq0}\delta_n\ell_n^\sigma\chi_{_{Q^n}}(x),\ \text{ for all }\  x\in\mathbb{R}^d.$$ By homogeneity of quasi-norms, we can assume that $\left\|F\right\|_{w,h}=1$. Thus, we have to show that $\left\|G\right\|_{u,v}\leq C$, where $C>0$ is a constant independent of the sequences $\left\{Q^n\right\}_{n\geq 0}$ and $\left\{\delta_n\right\}_{n\geq 0}$. For this purpose, we will consider separately cubes with measure less or equal to 1 and those with measure greater or equal to 2.

$1^{rst}$ case : we consider cubes $Q^m$ such that $|Q^m|\leq 1$. For each of them, we have
\begin{eqnarray*}
\sum_{k\in\mathbb{Z}^d}|Q^m\cap Q_k|^{\frac{h}{w}}=\left\|\chi_{_{Q^m}}\right\|_{w,h}^h\leq\frac{\left\|F\right\|_{w,h}^h}{\delta_m^h}=\frac{1}{\delta_m^h}\cdot
\end{eqnarray*}
Moreover, for all $k=(k_1,k_2,\ldots,k_d)\in\mathbb{Z}^d$, $Q_k$ is a dyadic cube of generation $0$ (we recall that $Q_k=\prod_{i=1}^d[k_i,k_i+1)$ ) and $\left\{Q_k\right\}_{k\in\mathbb{Z}^d}$ forms a partition of $\mathbb{R}^d$. Thus, since $Q^m$ is a dyadic cube of generation $\upsilon_m\geq0$, we have 
\begin{eqnarray*}
\frac{1}{\delta_m^h}\geq\left\|\chi_{_{Q^m}}\right\|_{w,h}^h&=&\sum_{\underset{Q^m\cap Q_k\neq\emptyset}{k\in\mathbb{Z}^d}}|Q^m\cap Q_k|^{\frac{h}{w}}\\
&=&|Q^m\cap Q_{k_0}|^{\frac{h}{w}} =|Q_m|^{\frac{h}{w}}=\ell_m^{d\frac{h}{w}},
\end{eqnarray*}
for some $k_0\in\mathbb{Z}^d$. Fix $x\in\mathbb R^d$. It comes that 
\begin{eqnarray}
\delta_m\ell_m^{\sigma}\chi_{_{Q^m}}(x)\leq\ell_m^{\sigma-\frac{d}{w}}. \label{0potenRiesz2}
\end{eqnarray}
Also, 
\begin{eqnarray}
\delta_m\ell_m^{\sigma}\chi_{_{Q^m}}(x)\leq\ell_m^{\sigma}\sum_{n\geq0}\delta_n\chi_{_{Q^n}}(x)=\ell_m^{\sigma}F(x). \label{0potenRiesz3}
\end{eqnarray}
Hence  
\begin{eqnarray*}
G(x)\leq\sum_{m\geq0}\min\left\{\ell_m^{\sigma}F(x),\ell_m^{\sigma-\frac{d}{w}}\right\},
\end{eqnarray*}
by (\ref{0potenRiesz2}) and (\ref{0potenRiesz3}). Since the cubes $\left\{Q^m\right\}_{m\geq0}$ are dyadic cubes with $Q^m\neq Q^{n}$ for $0\leq m<n<+\infty$, we obtain 
\begin{eqnarray*}
G(x)\leq C F(x)^{\frac{w}{u}}
\end{eqnarray*} 
with $C=\frac{1}{1-2^{\sigma-\frac{d}{w}}}+\frac{1}{1-2^{-\sigma}}>0$, according to Relation (\ref{potenRiesz}). It follows that 
\begin{eqnarray*}
\left\|G\right\|_{u,v}\leq C\left\|F^{\frac{w}{u}}\right\|_{u,v}=C\left\|F\right\|_{w,\frac{w}{u}v}^{\frac{w}{u}}.
\end{eqnarray*}
But, $h\leq\frac{w}{u}v$ by (\ref{potenRiesz}) and (\ref{0potenRiesz}). Hence 
\begin{eqnarray*}
\left\|G\right\|_{u,v}\leq C\left\|F\right\|_{w,\frac{w}{u}v}^{\frac{w}{u}}\leq C\left\|F\right\|_{w,h}^{\frac{w}{u}}=C.
\end{eqnarray*}

$2^{nd}$ case : we consider cubes $Q^m$ such that $|Q^m|\geq 2$. For each of these cubes, we have 
\begin{eqnarray*}
\sum_{k\in\mathbb{Z}^d}|Q^m\cap Q_k|^{\frac{h}{w}}=\left\|\chi_{_{Q^m}}\right\|_{w,h}^h\leq\frac{1}{\delta_m^h}\cdot
\end{eqnarray*}
Thus we have 
\begin{eqnarray*}
\frac{1}{\delta_m^h}\geq\left\|\chi_{_{Q^m}}\right\|_{w,h}^h&=&\sum_{\underset{Q^m\cap Q_k\neq\emptyset}{k\in\mathbb{Z}^d}}|Q^m\cap Q_k|^{\frac{h}{w}}\\
&=&\sum_{\underset{Q^m\cap Q_k\neq\emptyset}{k\in\mathbb{Z}^d}}|Q_k|^{\frac{h}{w}}=\ell_m^d,
\end{eqnarray*}
 since $Q^m$ is a dyadic cube of generation $\upsilon_m<0$. Therefore, for fixed $x\in\mathbb R^d$, we have
\begin{eqnarray}
\delta_m\ell_m^{\sigma}\chi_{_{Q^m}}(x)\leq\ell_m^{\sigma-\frac{d}{h}}. \label{0potenRiesz8}
\end{eqnarray}
Thus, 
\begin{eqnarray*}
G(x)\leq\sum_{m\geq0}\min\left\{\ell_m^{\sigma}F(x),\ell_m^{\sigma-\frac{d}{h}}\right\},
\end{eqnarray*}
by (\ref{0potenRiesz3}) and (\ref{0potenRiesz8}). Thanks to Relation (\ref{0potenRiesz}), we obtain as in the first case, 
\begin{eqnarray*}
G(x)\leq C F(x)^{\frac{h}{v}},
\end{eqnarray*}
with $C=\frac{1}{1-2^{-\sigma}}+\frac{1}{1-2^{\sigma-\frac{d}{h}}}>0$. Then, it follows that  
\begin{eqnarray*}
\left\|G\right\|_{u,v}\leq C\left\|F^{\frac{h}{v}}\right\|_{u,v}=C\left\|F\right\|_{\frac{h}{v}u,h}^{\frac{h}{v}}\leq C\left\|F\right\|_{w,h}^{\frac{h}{v}}=C, 
\end{eqnarray*}
since $\frac{h}{v}u\leq w$ according to Relations (\ref{potenRiesz}) and (\ref{0potenRiesz}). 

Combining both cases, we finally obtain 
\begin{eqnarray*}
\left\|G\right\|_{u,v}&\leq&C(u,v)\left(\left\|\sum_{n\geq0:\ \ell_n\leq1}\delta_n\ell_n^\sigma\chi_{_{Q^n}}\right\|_{u,v}+\left\|\sum_{n\geq0:\ \ell_n>1}\delta_n\ell_n^\sigma\chi_{_{Q^n}} \right\|_{u,v}\right)\\
&\leq&C(d,\sigma,w,h,u,v).
\end{eqnarray*}
The lemma is then proved in the case of dyadic cubes.

 For the general case, we consider for each cube $Q^n$ (not necessarily dyadic), a dyadic cube $R^n:=2^{-\upsilon_n}\prod_{i=1}^d[m_{n,i}\ ,\ m_{n,i}+1)$  such that $$10^d|Q^n|\geq|R^n|\ \ \text{ and }\ \ Q^n\subset 3R^n:=2^{-\upsilon_n}\prod_{i=1}^d[m_{n,i}-1\ ,\ m_{n,i}+2).$$ 
We have 
\begin{eqnarray*}
\left\|\sum_{n\geq0}\delta_n\ell_n^\sigma\chi_{_{Q^n}}\right\|_{u,v}&\leq&3^\sigma\left\|\sum_{n\geq0}\delta_n\ell(R^n)^\sigma\chi_{_{Q^n}}\right\|_{u,v}\\
&\leq&C(d,\gamma,\sigma)\left\|\sum_{n\geq0}\delta_n\ell(R^n)^\sigma\left[\mathfrak{M}\left(\chi_{_{R^n}}\right)\right]^\gamma\right\|_{u,v}\\
&\leq&C(d,u,v,\gamma,\sigma)\left\|\sum_{n\geq0}\delta_n\ell(R^n)^\sigma\chi_{_{R^n}}\right\|_{u,v},
\end{eqnarray*}
for all $\gamma>\max\left\{1,\frac{1}{u}\right\}$, by Proposition \ref{operamaxima}. Since (\ref{poten0Riesz}) is proved for dyadic cubes, we obtain 
\begin{eqnarray*}
\left\|\sum_{n\geq0}\delta_n\ell_n^\sigma\chi_{_{Q^n}}\right\|_{u,v}&\leq&C(d,\sigma,w,h,u,v)\left\|\sum_{n\geq0}\delta_n\chi_{_{R^n}}\right\|_{w,h}\\
&\leq&C(d,\sigma,w,h,u,v,\gamma)\left\|\sum_{n\geq0}\delta_n\left[\mathfrak{M}\left(\chi_{_{Q^n}}\right)\right]^\gamma\right\|_{w,h}\\
&\leq&C(d,\sigma,w,h,u,v)\left\|\sum_{n\geq0}\delta_n\chi_{_{Q^n}}\right\|_{w,h},
\end{eqnarray*}
where $\gamma=1+\frac{1}{w}$ , by Proposition \ref{operamaxima}.
This ends the proof of the lemma.   
\end{proof}

\begin{thm} \label{propofondaRiez1}
Let $0<q\leq\min(1,p)<+\infty$, $0<\alpha<d$ and $0<q_1\leq p_1<+\infty$ such that   
\begin{eqnarray}
\frac{1}{q_1}=\frac{1}{q}-\frac{\alpha}{d} \label{potenRiesz4}
\end{eqnarray}
and 
\begin{eqnarray}
\frac{1}{p_1}=\frac{1}{p}-\frac{\alpha}{d}\cdot \label{potenRiesz04}
\end{eqnarray}
Then, $I_\alpha$ extends to a bounded operator from $\mathcal{H}^{(q,p)}$ to $\mathcal{H}^{(q_1,p_1)}$.
\end{thm}
\begin{proof}
Let $\delta\geq\max\left\{\left\lfloor d\left(\frac{1}{q}-1\right)\right\rfloor;\alpha\right\}$ be an integer and $r>\max\left\{p_1;\frac{d}{d-\alpha}\right\}$ be a real. Let $\mathcal{H}_{fin}^{(q,p)}$ be the space of finite linear combinations of $(q,r,\delta)$-atoms and $f\in\mathcal{H}_{fin}^{(q,p)}$. Then, there exist a finite sequence $\left\{(\textbf{a}_n, Q^n)\right\}_{n=0}^j$ in $\mathcal{A}(q,r,\delta)$ and a finite sequence of scalars $\left\{{\lambda}_n\right\}_{n=0}^j$ such that $f=\sum_{n=0}^j\lambda_n\textbf{a}_n$. Set $\widetilde{Q^n}:=2\sqrt{d}Q^n$, $n\in\left\{0,1,\ldots,j\right\}$, and denote by $x_n$ and $\ell_n$ respectively the center and the side-length of $Q^n$. 

Now, we distinguish the cases: $1<q_1<+\infty$ and $0<q_1\leq 1$.

Case 1: $1<q_1<+\infty$. Then, $\mathcal{H}^{(q_1,p_1)}=(L^{q_1},\ell^{p_1})$ with norms equivalence, by \cite{AbFt}, Theorem 3.2, p. 1905. With this in mind, we have
\begin{eqnarray*}
|I_\alpha(f)(x)|\leq\sum_{n=0}^j|\lambda_n|\left(|I_\alpha(\textbf{a}_n)(x)|\chi_{_{\widetilde{Q^n}}}(x)+|I_\alpha(\textbf{a}_n)(x)|\chi_{_{\mathbb{R}^d\backslash{\widetilde{Q^n}}}}(x)\right),
\end{eqnarray*}
for almost all $x\in\mathbb{R}^d$. Hence 
\begin{eqnarray*}
\left\|I_\alpha(f)\right\|_{q_1,p_1}\leq\left\|\sum_{n=0}^j|\lambda_n||I_\alpha(\textbf{a}_n)|\chi_{_{\widetilde{Q^n}}}\right\|_{q_1,p_1}+\left\|\sum_{n=0}^j|\lambda_n||I_\alpha(\textbf{a}_n)|\chi_{_{\mathbb{R}^d\backslash{\widetilde{Q^n}}}}\right\|_{q_1,p_1}.
\end{eqnarray*}
Fix $0<\eta<q$. Set $$I=\left\|\sum_{n=0}^j|\lambda_n||I_\alpha(\textbf{a}_n)|\chi_{_{\widetilde{Q^n}}}\right\|_{q_1,p_1}\ \text{ and }\ J=\left\|\sum_{n=0}^j|\lambda_n||I_\alpha(\textbf{a}_n)|\chi_{_{\mathbb{R}^d\backslash{\widetilde{Q^n}}}}\right\|_{q_1,p_1}.$$ Since $r>\frac{d}{d-\alpha}$, there exists $1<s<r$ such that 
\begin{eqnarray}
\frac{1}{r}=\frac{1}{s}-\frac{\alpha}{d} \label{potenRiesz5}
\end{eqnarray}
and 
\begin{eqnarray}
\frac{1}{r}-\frac{1}{q_1}=\frac{1}{s}-\frac{1}{q}, \label{potenRiesz6}
\end{eqnarray}
by (\ref{potenRiesz4}). Thus, 
\begin{eqnarray*}
\left\|I_\alpha(\textbf{a}_n)\chi_{_{\widetilde{Q^n}}}\right\|_{r}\leq\left\|I_\alpha(\textbf{a}_n)\right\|_{r}
&\lsim&\left\|\textbf{a}_n\right\|_{s}\\
&\lsim&|Q^n|^{\frac{1}{s}-\frac{1}{q}} = C|Q^n|^{\frac{1}{r}-\frac{1}{q_1}}=C|\widetilde{Q^n}|^{\frac{1}{r}-\frac{1}{q_1}}\ ,
\end{eqnarray*}
by (\ref{potenRiesz5}), the fact that $I_\alpha$ is bounded from $L^s$ to $L^r$, $\mathcal{A}(q,r,\delta)\subset\mathcal{A}(q,s,\delta)$ and (\ref{potenRiesz6}). Furthermore, $1<q_1\leq p_1<r<+\infty$ and $\text{supp}(|I_\alpha(\textbf{a}_n)|\chi_{_{\widetilde{Q^n}}})\subset\widetilde{Q^n}$. Hence 
\begin{eqnarray*}
I\lsim\left\|\sum_{n=0}^j\frac{|\lambda_n|}{\left\|\chi_{_{\widetilde{Q^n}}}\right\|_{q_1}}\chi_{_{\widetilde{Q^n}}}\right\|_{q_1,p_1}\lsim\left\|\sum_{n=0}^j\frac{|\lambda_n|}{\left\|\chi_{_{Q^n}}\right\|_{q_1}}\chi_{_{Q^n}}\right\|_{q_1,p_1},
\end{eqnarray*}
by Proposition 4.5 in \cite{AbFt} and the proof of Theorem 4.6 in  \cite{AbFt}. But, 
\begin{eqnarray}
\left\|\chi_{_{Q^n}}\right\|_{q_1}=|Q^n|^{\frac{1}{q_1}}=|Q^n|^{\frac{1}{q}-\frac{\alpha}{d}}=\left\|\chi_{_{Q^n}}\right\|_{q}\ell_n^{-\alpha}.\label{0potenRiesz06}
\end{eqnarray}
Hence  
\begin{eqnarray*}
I\lsim\left\|\sum_{n=0}^j\frac{|\lambda_n|}{\left\|\chi_{_{Q^n}}\right\|_{q}}\chi_{_{Q^n}}\right\|_{q,p}\ , 
\end{eqnarray*}
by Lemma \ref{propofondaRiez}. To estimate $J$, notice that
\begin{eqnarray}
|\partial^{\beta}K_\alpha(x)|\leq \frac{C_{d,\alpha,\delta}}{|x|^{d-\alpha+|\beta|}}\ , \label{potenRiesz8}
\end{eqnarray}
for all $x\in\mathbb{R}^d\backslash\left\{0\right\}$ and all multi-indices $\beta$ with $|\beta|\leq\delta+1$, where $C_{d,\alpha,\delta}>0$ is a constant independent of $\beta$ and $x$. Consider $x\notin\widetilde{Q^n}$. We have $x-y\neq0$, for all $y\in Q^n$. Thus, using Taylor's formula and the vanishing condition of the atom $\textbf{a}_n$ , we obtain 
\begin{eqnarray*}
|I_\alpha(\textbf{a}_n)(x)|=\left|\int_{Q^n}K_\alpha(x-y)\textbf{a}_n(y)dy\right|\leq\int_{Q^n}|R_\alpha(x-y,x-x_n)||\textbf{a}_n(y)|dy,
\end{eqnarray*}
with
\begin{eqnarray*}
|R_\alpha(x-y,x-x_n)|&=&\left|\delta'\sum_{|\beta|=\delta'}\frac{(x_n-y)^{\beta}}{\beta!}\int_{0}^1(1-\theta)^{\delta}\partial^{\beta}K_\alpha((x-x_n)+\theta(x_n-y))d\theta\right|\\
&\lsim&\frac{|x_n-y|^{\delta+1}}{|x-x_n|^{d-\alpha+\delta+1}}\ ,
\end{eqnarray*}
for all $y\in Q^n$, where $\delta'=\delta+1$, by (\ref{potenRiesz8}). Hence 
\begin{eqnarray*}
|I_\alpha(\textbf{a}_n)(x)|&\lsim&\int_{Q^n}\frac{|x_n-y|^{\delta+1}}{|x-x_n|^{d-\alpha+\delta+1}}|\textbf{a}_n(y)|dy\\
&\lsim&\frac{\ell_n^{d-\alpha+\delta+1}}{|x-x_n|^{d-\alpha+\delta+1}}\frac{\ell_n^\alpha}{\left\|\chi_{_{Q^n}}\right\|_q}\\ 
&\lsim&\ell_n^\alpha\frac{\left[\mathfrak{M}(\chi_{_{Q^n}})(x)\right]^{\vartheta}}{\left\|\chi_{_{Q^n}}\right\|_q}\ ,
\end{eqnarray*}
where $\vartheta=\frac{d-\alpha+\delta+1}{d}\cdot$ Moreover, since $\delta\geq\max\left\{\left\lfloor d\left(\frac{1}{q}-1\right)\right\rfloor;\alpha\right\}$, we have $1<\vartheta$ and $1<q_1\vartheta\leq p_1\vartheta<+\infty$. Therefore,
\begin{eqnarray*}
J&\lsim& \left\|\sum_{n=0}^j|\lambda_n|\ell_n^\alpha\frac{\left[\mathfrak{M}(\chi_{_{Q^n}})\right]^{\vartheta}}{\left\|\chi_{_{Q^n}}\right\|_q}\right\|_{q_1,p_1}\\
&\lsim&\left\|\sum_{n=0}^j\frac{|\lambda_n|}{\left\|\chi_{_{Q^n}}\right\|_q}\ell_n^\alpha\chi_{_{Q^n}}\right\|_{q_1,p_1}\lsim\left\|\sum_{n=0}^j\frac{|\lambda_n|}{\left\|\chi_{_{Q^n}}\right\|_q}\chi_{_{Q^n}}\right\|_{q,p}\ , 
\end{eqnarray*}
by the proof of Theorem 4.6 in  \cite{AbFt} and Lemma \ref{propofondaRiez}. Finally, we obtain  
\begin{eqnarray*}
\left\|I_\alpha(f)\right\|_{q_1,p_1}\lsim\left\|\sum_{n=0}^j\frac{|\lambda_n|}{\left\|\chi_{_{Q^n}}\right\|_q}\chi_{_{Q^n}}\right\|_{q,p}\lsim\left\|\sum_{n=0}^j\left(\frac{|\lambda_n|}{\left\|\chi_{_{Q^n}}\right\|_q}\right)^{\eta}\chi_{_{Q^n}}\right\|_{\frac{q}{\eta},\frac{p}{\eta}}^{\frac{1}{\eta}}\cdot
\end{eqnarray*}
Thus,
\begin{eqnarray*}
\left\|I_\alpha(f)\right\|_{q_1,p_1}\lsim\left\|f\right\|_{\mathcal{H}_{fin}^{(q,p)}}\lsim\left\|f\right\|_{\mathcal{H}^{(q,p)}},
\end{eqnarray*}
by Theorem \ref{thafondamfini}, since (\ref{potenRiesz04}) implies that $r>p_1>p$. Hence $I_\alpha$ is bounded from $\mathcal{H}_{fin}^{(q,p)}$ to $(L^{q_1},\ell^{p_1})$. The density of $\mathcal{H}_{fin}^{(q,p)}$ in $\mathcal{H}^{(q,p)}$ and the fact that $(L^{q_1},\ell^{p_1})=\mathcal{H}^{(q_1,p_1)}$ with norms equivalence yield the result.

Case 2: $0<q_1\leq 1$. We have
\begin{eqnarray*}
\mathcal{M}_{0}(I_\alpha(f))(x)\leq\sum_{n=0}^j|\lambda_n|\left[\mathcal{M}_{0}(I_\alpha(\textbf{a}_n))(x)\chi_{_{\widetilde{Q^n}}}(x)+\mathcal{M}_{0}(I_\alpha(\textbf{a}_n))(x)\chi_{_{\mathbb{R}^d\backslash{\widetilde{Q^n}}}}(x)\right].
\end{eqnarray*}
for all $x\in\mathbb{R}^d$. Hence 
\begin{eqnarray*}
\left\|\mathcal{M}_{0}(I_\alpha(f))\right\|_{q_1,p_1}\lsim I+J 
\end{eqnarray*}
with $$I=\left\|\sum_{n=0}^j|\lambda_n|\mathcal{M}_{0}(I_\alpha(\textbf{a}_n))\chi_{_{\widetilde{Q^n}}}\right\|_{q_1,p_1}\ \text{ and }\ J=\left\|\sum_{n=0}^j|\lambda_n|\mathcal{M}_{0}(I_\alpha(\textbf{a}_n))\chi_{_{\mathbb{R}^d\backslash{\widetilde{Q^n}}}}\right\|_{q_1,p_1}.$$ Fix $0<\eta<q$. Let us consider  $1<s<r$ as in (\ref{potenRiesz5}). We have
\begin{eqnarray*}
\left\|\left(\mathcal{M}_{0}(I_\alpha(\textbf{a}_n))\chi_{_{\widetilde{Q^n}}}\right)^{\eta}\right\|_{\frac{r}{\eta}}
&\leq&\left\|\mathcal{M}_{0}(I_\alpha(\textbf{a}_n))\right\|_{r}^{\eta}\\
&\lsim&\left\|\mathfrak{M}(I_\alpha(\textbf{a}_n))\right\|_{r}^{\eta}\\
&\lsim&\left\|I_\alpha(\textbf{a}_n)\right\|_{r}^{\eta}\\
&\lsim&\left\|\textbf{a}_n\right\|_{s}^{\eta}\lsim|Q^{n}|^{\left(\frac{1}{s}-\frac{1}{q}\right)\eta}=C|\widetilde{Q^n}|^{\frac{1}{\frac{r}{\eta}}-\frac{1}{\frac{q_1}{\eta}}},
\end{eqnarray*}
by (\ref{potenRiesz6}). Moreover, $\text{supp}\left(\mathcal{M}_{0}(I_\alpha(\textbf{a}_n))\chi_{_{\widetilde{Q^n}}}\right)^{\eta}\subset\widetilde{Q^n}$ and $1<\frac{q_1}{\eta}\leq\frac{p_1}{\eta}<\frac{r}{\eta}\cdot$
Hence
\begin{eqnarray*}
I&\leq&\left\|\sum_{n=0}^j|\lambda_n|^{\eta}\left(\mathcal{M}_{0}(I_\alpha(\textbf{a}_n))\chi_{_{\widetilde{Q^n}}}\right)^{\eta}\right\|_{\frac{q_1}{\eta},\frac{p_1}{\eta}}^{\frac{1}{\eta}}\lsim\left\|\sum_{n=0}^j\frac{|\lambda_n|^\eta}{\left\|\chi_{_{\widetilde{Q^n}}}\right\|_{\frac{q_1}{\eta}}}\chi_{_{\widetilde{Q^n}}}\right\|_{\frac{q_1}{\eta},\frac{p_1}{\eta}}^{\frac{1}{\eta}}\\ 
&\lsim& \left\|\sum_{n=0}^j\left(\frac{|\lambda_n|}{\left\|\chi_{_{Q^n}}\right\|_{q_1}}\right)^\eta\chi_{_{Q^n}}\right\|_{\frac{q_1}{\eta},\frac{p_1}{\eta}}^{\frac{1}{\eta}}\lsim\left\|\sum_{n=0}^j\left(\frac{|\lambda_n|}{\left\|\chi_{_{Q^n}}\right\|_{q}}\right)^\eta\ell_n^{\alpha\eta}\chi_{_{Q^n}}\right\|_{\frac{q_1}{\eta},\frac{p_1}{\eta}}^{\frac{1}{\eta}},
\end{eqnarray*}
by Proposition 4.5 in \cite{AbFt}, the proof of Theorem 4.6 in  \cite{AbFt} and (\ref{0potenRiesz06}). Set $\rho=\alpha\eta$ , $u=\frac{q_1}{\eta}$ ; $v=\frac{p_1}{\eta}$ ; $w=\frac{q}{\eta}$ and $h=\frac{p}{\eta}\cdot$ We have $0<\rho<d$ , $0<w\leq h<+\infty$ and $0<u\leq v<+\infty$. Furthermore, 
\begin{eqnarray*}
\frac{1}{u}=\frac{1}{w}-\frac{\rho}{d}\ \ \text{ and }\ \ \frac{1}{v}=\frac{1}{h}-\frac{\rho}{d}\ , 
\end{eqnarray*}
by (\ref{potenRiesz4}) and (\ref{potenRiesz04}). Thus, setting $\delta_n=\left(\frac{|\lambda_n|}{\left\|\chi_{_{Q^n}}\right\|_{q}}\right)^\eta$, for $n=0,1,\ldots,j$ , we have 
\begin{eqnarray*}
\left\|\sum_{n=0}^j\left(\frac{|\lambda_n|}{\left\|\chi_{_{Q^n}}\right\|_{q}}\right)^\eta\ell_n^{\alpha\eta}\chi_{_{Q^n}}\right\|_{\frac{q_1}{\eta},\frac{p_1}{\eta}}&=&\left\|\sum_{n=0}^j\delta_n\ell_n^{\rho}\chi_{_{Q^n}}\right\|_{u,v}\lsim\left\|\sum_{n=0}^j\delta_n\chi_{_{Q^n}}\right\|_{w,h}\\
&=&C\left\|\sum_{n=0}^j\left(\frac{|\lambda_n|}{\left\|\chi_{_{Q^n}}\right\|_{q}}\right)^\eta\chi_{_{Q^n}}\right\|_{\frac{q}{\eta},\frac{p}{\eta}},
\end{eqnarray*}
by Lemma \ref{propofondaRiez}. Hence 
\begin{eqnarray}
I\lsim\left\|\sum_{n=0}^j\left(\frac{|\lambda_n|}{\left\|\chi_{_{Q^n}}\right\|_{q}}\right)^\eta\chi_{_{Q^n}}\right\|_{\frac{q}{\eta},\frac{p}{\eta}}^{\frac{1}{\eta}}. \label{potenRiesz11}
\end{eqnarray}
To estimate $J$, it suffices to show that 
\begin{eqnarray}
\mathcal{M}_{0}(I_\alpha(\textbf{a}_n))(x)\lsim\ell_n^\alpha\frac{\left[\mathfrak{M}(\chi_{_{Q^n}})(x)\right]^{\vartheta}}{\left\|\chi_{_{Q^n}}\right\|_q}\ , \label{potenRiesz12}
\end{eqnarray}
for all $x\notin\widetilde{Q^n}$, where $\vartheta=\frac{d-\alpha+\delta+1}{d}\cdot$ To prove (\ref{potenRiesz12}), we consider $x\notin\widetilde{Q^n}$. For every $t>0$, set $K_\alpha^{(t)}=K_\alpha\ast\varphi_t$. We have 
\begin{eqnarray}
\sup_{t>0}|\partial^{\beta}_z K_\alpha^{(t)}(z)|\leq\frac{C_{\varphi,d,\alpha,\delta}}{|z|^{d-\alpha+\delta+1}}\ , \label{potenRiesz13}
\end{eqnarray}
for all $z\in\mathbb{R}^d\backslash\left\{0\right\}$ and all multi-indexes $\beta$ with $|\beta|=\delta+1$. The proof of (\ref{potenRiesz13}) is similar to the one of (\ref{applicattheo19}) in the proof of Theorem \ref{theoremsing5}; therefore, details are left to the reader. With (\ref{potenRiesz13}), by using Taylor's formula and the vanishing condition of the atom $\textbf{a}_n$, we have obtain, as in Case 1,
\begin{eqnarray*}
|(I_\alpha(\textbf{a}_n)\ast\varphi_t)(x)|&=&|\textbf{a}_n\ast K_\alpha^{(t)}(x)|=\left|\int_{Q^n}K_\alpha^{(t)}(x-y)\textbf{a}_n(y)dy\right|\\
&\leq&\int_{Q^n}|R_\alpha^{(t)}(x-y,x-x_n)||\textbf{a}_n(y)|dy\\
&\lsim&\int_{Q^n}\frac{|x_n-y|^{\delta+1}}{|x-x_n|^{d-\alpha+\delta+1}}|\textbf{a}_n(y)|dy\lsim\ell_n^\alpha\frac{\left[\mathfrak{M}(\chi_{_{Q^n}})(x)\right]^{\vartheta}}{\left\|\chi_{_{Q^n}}\right\|_q}\ ,
\end{eqnarray*}
where $\vartheta=\frac{d-\alpha+\delta+1}{d}\cdot$ Hence
\begin{eqnarray*}
\mathcal{M}_{0}(I_\alpha(\textbf{a}_n))(x)\lsim\ell_n^\alpha\frac{\left[\mathfrak{M}(\chi_{_{Q^n}})(x)\right]^{\vartheta}}{\left\|\chi_{_{Q^n}}\right\|_q}\ ,
\end{eqnarray*}
with $\vartheta=\frac{d-\alpha+\delta+1}{d}\cdot$ This establishes (\ref{potenRiesz12}).\\
Moreover, since $\delta\geq\max\left\{\left\lfloor d\left(\frac{1}{q}-1\right)\right\rfloor;\alpha\right\}$, we have $1<\vartheta$ and $1<q_1\vartheta\leq p_1\vartheta<+\infty$, by (\ref{potenRiesz4}). Thus, as in Case 1, we obtain 
\begin{eqnarray*}
J&\lsim&\left\|\sum_{n=0}^j|\lambda_n|\ell_n^\alpha\frac{\left[\mathfrak{M}(\chi_{_{Q^n}})\right]^{\vartheta}}{\left\|\chi_{_{Q^n}}\right\|_q}\right\|_{q_1,p_1}\\
&\lsim&\left\|\sum_{n=0}^j\frac{|\lambda_n|}{\left\|\chi_{_{Q^n}}\right\|_q}\chi_{_{Q^n}}\right\|_{q,p}\lsim \left\|\sum_{n=0}^j\left(\frac{|\lambda_n|}{\left\|\chi_{_{Q^n}}\right\|_q}\right)^{\eta}\chi_{_{Q^n}}\right\|_{\frac{q}{\eta},\frac{p}{\eta}}^{\frac{1}{\eta}}\cdot 
\end{eqnarray*}
Hence 
\begin{eqnarray*}
\left\|\mathcal{M}_{0}(I_\alpha(f))\right\|_{q_1,p_1}\lsim\left\|\sum_{n=0}^j\left(\frac{|\lambda_n|}{\left\|\chi_{_{Q^n}}\right\|_q}\right)^{\eta}\chi_{_{Q^n}}\right\|_{\frac{q}{\eta},\frac{p}{\eta}}^{\frac{1}{\eta}}\cdot
\end{eqnarray*}
Thus,
\begin{eqnarray*}
\left\|I_\alpha(f)\right\|_{\mathcal{H}^{(q_1,p_1)}}=\left\|\mathcal{M}_{0}(I_\alpha(f))\right\|_{q_1,p_1}\lsim 
\left\|f\right\|_{\mathcal{H}_{fin}^{(q,p)}}\lsim\left\|f\right\|_{\mathcal{H}^{(q,p)}},
\end{eqnarray*}
by Theorem \ref{thafondamfini}, since (\ref{potenRiesz04}) implies that $r>p_1>p$. Therefore, $I_\alpha$ is bounded from $\mathcal{H}_{fin}^{(q,p)}$ to $\mathcal{H}^{(q_1,p_1)}$ and the density of $\mathcal{H}_{fin}^{(q,p)}$ in $\mathcal{H}^{(q,p)}$ yields the result. This ends the proof of Theorem \ref{propofondaRiez1}.  
\end{proof}

\begin{remark}
The proof of Theorem \ref{propofondaRiez1} in the second case ($0<q_1\leq 1$) is still valid whenever $1<q_1<+\infty$.
\end{remark}
We mention that one can use molecular characterization of $\mathcal H^{(q_{1},p_{1})}$ in the second case ($0<q_{1}\leq 1$) as defined in \cite{AbFt}, for the proof. 

\begin{cor}
$I_\alpha$ extends to a bounded operator from $\mathcal{H}^{(q,p)}$ to $\mathcal{H}^{(q_2,p_2)}$, if 
\begin{eqnarray}
\frac{1}{q_2}\geq\frac{1}{q}-\frac{\alpha}{d}\ \text{ and }\ \frac{1}{p_2}\leq\frac{1}{p}-\frac{\alpha}{d}\cdot \label{potenRiesz4bis}
\end{eqnarray}
\end{cor}
\begin{proof}
(\ref{potenRiesz4bis}) implies that $q_2\leq q_1$ and $p_1\leq p_2$. Hence, by 
 Theorem \ref{propofondaRiez1}, 
$$\left\|I_\alpha(f)\right\|_{\mathcal{H}^{(q_2,p_2)}}\leq\left\|I_\alpha(f)\right\|_{\mathcal{H}^{(q_2,p_1)}}\leq\left\|I_\alpha(f)\right\|_{\mathcal{H}^{(q_1,p_1)}}\lsim\left\|f\right\|_{\mathcal{H}^{(q,p)}},$$
 for all $f\in\mathcal{H}_{fin}^{(q,p)}$, where $\mathcal{H}_{fin}^{(q,p)}$ is the subspace of $\mathcal{H}^{(q,p)}$ consisting of finite linear combinations of $(q,r,\delta)$-atoms with $\delta\geq\max\left\{\left\lfloor d\left(\frac{1}{q}-1\right)\right\rfloor;\alpha\right\}$ and $\max\left\{p_1;\frac{d}{d-\alpha}\right\}<r<+\infty$ fixed. Then, the density of $\mathcal{H}_{fin}^{(q,p)}$ in $\mathcal{H}^{(q,p)}$ yields the result.
\end{proof}

\end{document}